\documentclass{amsart}
\usepackage{amsaddr}
\usepackage{amsmath}
\usepackage{amssymb}
\usepackage{amsthm}
\usepackage{yfonts}

\usepackage[]{hyperref}

\usepackage{graphicx}
\usepackage{overpic}

\newtheorem{theorem}{Theorem}[section]
\newtheorem{lemma}[theorem]{Lemma}
\newtheorem{proposition}[theorem]{Proposition}
\newtheorem{corollary}[theorem]{Corollary}

\newtheorem{subclaim}[theorem]{Sub-Claim}

\newtheorem{introtheorem}{Theorem}

\theoremstyle{definition}
\newtheorem{definition}[theorem]{Definition}

\theoremstyle{remark}
\newtheorem{remark}{Remark}[theorem]

\usepackage{xcolor}

\usepackage{marginnote}
\newcounter{mynote}

\title{Comparing the Roller and B(X) boundaries of CAT(0) cube complexes}
\author{Ivan Levcovitz}
\email{Levcovitz@technion.ac.il}
\address{Technion - Israel Institute of Technology
\\Department of Mathematics
\\Haifa, 32000 
\\Israel
}
\thanks{This work was supported by the Israel Science Foundation and in part by a Technion fellowship.}
\begin{document}
\maketitle
\begin{abstract}
The Roller boundary is a well-known compactification of a CAT(0) cube complex $X$. When $X$ is locally finite, essential, irreducible, non-Euclidean and admits a cocompact action by a group $G$, Nevo-Sageev show that a subset, $B(X)$, of the Roller boundary is the realization of the Poisson boundary and that the action of $G$ on $B(X)$ is minimal and strongly proximal. Additionally, these authors show $B(X)$ satisfies many other desirable dynamical and topological properties. In this article we give several equivalent characterizations for when $B(X)$ is equal to the entire Roller boundary. As an application
we show, under mild hypotheses, that if $X$ is also $2$--dimensional then $X$ is $G$--equivariantly quasi-isometric to a CAT(0) cube complex $X'$ whose Roller boundary is equal to $B(X')$. Additionally, we use our characterization to show that the usual CAT(0) cube complex for which an infinite right-angled Coxeter/Artin group acts on geometrically has Roller boundary equal to $B(X)$, as long as the corresponding group does not decompose as a direct product.
\end{abstract}
\section{Introduction}
CAT(0) cube complexes have played a central role in geometric group theory and low-dimensional topology. For instance, the resolution of the virtual Haken conjecture, an outstanding conjecture of Thurston, by Agol \cite{Ago} and Wise \cite{Wise}, relied heavily on CAT(0) cube complex developments. The class of groups that act nicely on a CAT(0) cube complex is surprisingly large, including Coxeter groups \cite{NR}, right-angled Artin groups, small cancellation groups \cite{Wise_sc}, several classes of Artin groups \cite{Haettel, HJP} and random groups at density less than $\frac{1}{6}$ \cite{OW}. General criteria are given by Sageev \cite{Sag} for when one can obtain such actions. 

Associated to a CAT(0) cube complex, $X$, there is a natural compactification introduced by Roller \cite{Rol} now known as the Roller boundary, $\partial X$. As a set, the Roller boundary consists of ultrafilters on halfspaces of $X$ which are ``at infinity.'' The Roller boundary has been well-studied and has proven useful for tackling several different problems regarding cube complexes \cite{CFI, Fernos, Hag2, NS, Rol}.

The Poisson boundary of a group, introduced in  \cite{Furstenberg}, is roughly the space of all possible directions at infinity a random walk can take on the given group. In \cite{NS}, Nevo-Sageev single out a special subspace, $B(X)$, of the Roller boundary. These authors show that if $X$ is locally finite, essential, irreducible, non-Euclidean and admits a cocompact group action, then $B(X)$ is a minimal realization of the Poisson boundary. Furthermore, these authors show $B(X)$  satisfies many additional interesting dynamical and topological properties: $B(X)$ is a compact, metric, minimal, strongly-proximal, uniquely-stationary, mean-proximal, universally amenable and equicontinuously decomposable realization of the Poisson boundary. In \cite{Fernos} and \cite{FLF} some of these results are generalized to more general cube complexes and a different perspective on Poisson boundaries of CAT(0) cube complex is given. Notably, in \cite{FLF} a new characterization of the subset $B(X)$ is given in terms of certain rays in $X$.

In this article we characterize when $B(X)$ is equal to the entire Roller boundary, $\partial X$. In fact, we show that in many well studied settings either this equality holds or one can consider an appropriate alternative complex for which it holds. Consequently, often one can take the Roller boundary itself as a minimal Poisson boundary.

We give two equivalent conditions for when $\partial X = B(X)$. The first of which, the property of having caged hyperplanes, is a condition concerning finite subsets of halfspaces of $X$. Roughly, $X$ has caged hyperplanes if given any vertex of $X$, the intersection of every appropriate set of halfspaces close to the given vertex contains a hyperplane. The second characterization is in terms of open sets in the space of ultrafilters on $X$.

\begin{introtheorem}[Theorem \ref{thm_char}] \label{intro_thm_char}
Let $X$ be an essential, locally finite, cocompact CAT(0) cube complex. The following are equivalent:
	\begin{enumerate}
		\item $\partial X = B(X)$
		\item $X$ has caged hyperplanes.
		\item Every open set in the space of ultrafilters on $X$, which contains an ultrafilter in $\partial X$, contains a hyperplane. 
	\end{enumerate}
\end{introtheorem}

We note that the essential hypothesis is not a heavy requirement, as one can pass to an invariant essential subcomplex \cite{CS}. 

We next focus our attention to CAT(0) cube complexes with straight links, a generalization of extendable CAT(0) geodesics. Here stronger results are possible. In this setting we show that $\partial X \neq B(X)$ if and only if $X$ contains a certain finite set of halfspaces, which we call a tight cage (Definition \ref{def_tight_cage}), with a very distinctive structure. Using this extra structure we prove that given such a CAT(0) cube complex, $X$, that is $2$--dimensional and irreducible, and given a proper, cocompact action, satisfying a mild hypothesis, by a group $G$, then $X$ can be $G$--equivariantly replaced by a quasi-isometric CAT(0) cube complex $\bar{X}$ for which $\partial \bar{X} = B(\bar{X} )$:

\begin{introtheorem}[Theorem \ref{thm_2_dim_ccc}] \label{intro_thm_2d}
	Let $X$ be a $2$--dimensional, irreducible, locally finite CAT(0) cube complex with straight links. Suppose $G$ acts properly, cocompactly and without core carrier reflections on $X$. Then $G$ acts properly and cocompactly on a $2$--dimensional, irreducible, locally finite CAT(0) cube complex with straight links, $\bar{X}$, which satisfies $\partial \bar{X} = B(\bar{X})$. Furthermore, $X$ is $G$--equivariantly quasi-isometric to $\bar{X}$.
\end{introtheorem}

We briefly discuss the above assumption on the group's action. Let $\hat{h}$ be a hyperplane in $X$ and $\hat{h} \times [0,1]$ denote its carrier. Let $h$ be the halfspace associated to $\hat{h}$ which contains $\hat{h} \times \{1\} \subset \hat{h} \times [0,1]$. An element $g \in G$ acts as a carrier reflection on $h$, if it stabilizes $\hat{h} \times \{1\}$ and does not stabilize $\hat{h}$. The group $G$ acts without carrier reflections if given any $g \in G$ and any halfspace $h$, it follows that $g$ does not act as a carrier reflection of $h$. The acts without ``core'' carrier reflections hypothesis above is a weakening of the acting without carrier reflections hypothesis.

For general cocompact, locally finite CAT(0) cube complexes with straight links, we show it is often a straightforward task to recognize when $B(X)$ is equal to the whole Roller boundary. For instance, the following gives a criterion for doing so:

\begin{introtheorem}[Corollary \ref{cor_convex_join_subset}] \label{intro_thm_joins}
Let $X$ be a cocompact, locally finite CAT(0) cube complex with straight links. If $\partial X \neq B(X)$, then $X$ contains an unbounded convex subset $Y$ such that the link, taken in $X$, of every vertex in $Y$ is a join.
\end{introtheorem} 

We apply this result to infinite right-angled Coxeter groups and right-angled Artin group. It follows that, as long such a given group is not a direct product, then the usual CAT(0) cube complex, $X$, it acts on satisfies $B(X) = \partial X$.

\begin{introtheorem}[Theorem \ref{thm_racg_app}, Theorem \ref{thm_raag_app}] \label{intro_thm_racg_raag}
	Let $X$ either be the Davis complex of an infinite right-angled Coxeter group or the universal cover of the Salvetti complex of a right-angled Artin group. Then $B(X) = \partial X$ if and only if the corresponding right-angled Coxeter/Artin group does not split as a direct product.
\end{introtheorem}

The outline of this article is as follows. Section \ref{sec_background} gives the relevant background regarding CAT(0) cube complexes and the boundaries we consider. After this, we introduce the notion of caged hyperplanes in Section \ref{sec_caged_hyps}. In particular, some implications in Theorem \ref{intro_thm_char} are shown. Section \ref{sec_straight_links} concerns the straight links hypothesis. There it is shown that CAT(0) cube complexes with straight links have many desirable structural properties. The following section, Section \ref{sec_tight_cages}, introduces the notion of tight cages. Tight cages are used to give an additional, very useful, characterization of when the Roller boundary is equal to $B(X)$ when $X$ has straight links. 

As Euclidean and reducible CAT(0) cube complexes behave fundamentally differently than their counterparts, we treat these separately in Section \ref{sec_euc_red}. In Section \ref{sec_characterization} we compile the results from previous sections to prove the main characterization result, Theorem \ref{intro_thm_char}. As an application, we describe a straightforward approach to show the Roller boundary is equal to $B(X)$ in Section \ref{sec_applications} which works for many well studied CAT(0) cube complexes. We illustrate this approach by proving Theorem \ref{intro_thm_racg_raag}. The final section is dedicated to proving Theorem \ref{intro_thm_2d}.

\subsection*{Acknowledgments:}
I am deeply thankful to Michah Sageev for directing me towards this area of research and for the many fruitful discussions regarding this work. I would also like to thank Nir Lazarovich for helpful discussions regarding cube complexes. I am very thankful to the anonymous referee for excellent suggestions which led to the simplification of several arguments and a strengthening of Proposition \ref{prop_racg_straight_links}.

\section{Background} \label{sec_background}

\subsection{CAT(0) cube complexes}

A \textit{CAT(0) cube complex}, $X$, is a simply connected cell complex whose cells consist of Euclidean unit cubes, $[ -\frac{1}{2}, \frac{1}{2}]^d$, of varying dimension $d$. Additionally, the link of each vertex is a flag complex, i.e., any set of vertices which are pairwise connected by an edge, spans a simplex. All CAT(0) cube complexes considered in this article are assumed to be connected. We refer the reader to \cite{CS} and \cite{Wise} for a detailed background on cube complexes. All facts stated in this background regarding cube complexes can be found in these references.

\subsubsection{Definitions}
We say $X$ is \textit{finite-dimensional} if there is an upper bound on the dimension of cubes in $X$. We say $X$ is \textit{locally finite}, if every vertex in the $1$--skeleton of $X$ has finite valence. Furthermore, we say $X$ is \textit{cocompact}, if the group of isometries of $X$ acts cocompactly on $X$.

A \textit{midcube} of a cube, $[ -\frac{1}{2}, \frac{1}{2}]^d$, is the restriction of a coordinate of the cube to 0. A \textit{hyperplane}, $\hat{h}$, is a connected subspace of $X$ with the property that for each cube $C$ in $X$, $\hat{h} \cap C$ is a midcube or $\hat{h} \cap C = \emptyset$. If $\hat{h} \cap e \neq \emptyset$ for some edge $e$ in $X$, then we say $\hat{h}$ is dual to $e$.  It is a basic fact in the theory of CAT(0) cube complexes that $\hat{h}$ has itself the structure of a CAT(0) cube complex where each midcube in $\hat{h}$ is considered as a cube. If $X$ is finite dimensional, the cube complex corresponding to $\hat{h}$ has strictly smaller dimension than $X$.

Given any hyperplane $\hat{h}$ in $X$, $X \setminus \hat{h}$ consists of exactly two distinct components. The closure of such a component is called a \textit{halfspace}. We denote the two halfspaces associated to $\hat{h}$ by $h$ and $h^*$. Furthermore, we say that $h$ is a choice of orientation for $\hat{h}$. We take the following convention for notation: we always designate halfspaces by lowercase letters (e.g., $h$) and denote their corresponding hyperplanes by the same hatted letter (e.g., $\hat{h}$). Similarly, if $H$ is a collection of halfspaces, then $\hat{H}$ will always denote the corresponding set of hyperplanes. Two halfspaces, $h$ and $k$ are \textit{comparable} if either $h \subset k$ or $k \subset h$. Otherwise, $h$ and $k$ are \textit{incomparable}.

Let $\hat{h}$ be a hyperplane and $h$ a choice of halfspace for $\hat{h}$. The \textit{carrier}, $\mathcal{C}(\hat{h})$, of the hyperplane $\hat{h}$ is the set of all cubes in $X$ that have non-trivial intersection with $\hat{h}$. It follows that $\mathcal{C}(\hat{h})$ is isometric to $\hat{h} \times I$, where $I = [0,1]$. The subcomplexes, $\hat{h} \times \{0\}$ and $\hat{h} \times \{1\}$, of $\mathcal{C}(\hat{h})$ are each isometric to $\hat{h}$ and are each contained in a distinct component of $X \setminus \hat{h}$. We assume our labeling is such that $\hat{h} \times \{1\}$ is contained in $h$ and that $\hat{h} \times \{0\}$ is contained in $h^*$. We let $\mathcal{C}^+(h)$ denote $\hat{h} \times \{1\} \subset \mathcal{C}(\hat{h})$, and we say that $\mathcal{C}^+(h)$ the \textit{positive carrier of $h$}. Similarly, we let $\mathcal{C}^+(h^*)$ denote $\hat{h} \times \{0\} \subset \mathcal{C}(\hat{h})$ and say it is the positive carrier of $h^*$.

In this article we will exclusively work with the combinatorial metric on the 1--skeleton of $X$. By a \textit{path} in $X$, we mean a path in the $1$--skeleton of $X$ consisting of a sequence of edges. The length of a path is defined to be the number of edges in the path. A \textit{geodesic} in $X$ is a path of minimal length out of all possible paths with the same endpoints. It is an important fact that a path in $X$ is geodesic if and only if every hyperplane in $X$ is dual to at most one edge of the path.

Finally, we say that $X$ is \textit{essential}, if given any halfspace $h$ in $X$, there are vertices in $h$ arbitrarily far from $\hat{h}$.

\subsubsection{Convexity}

Let $X$ be a CAT(0) cube complex. A subcomplex $Y \subset X$ is \textit{convex}, if every geodesic between two vertices of $Y$ is contained in $Y$. An important example of a convex subcomplex is the carrier of a hyperplane in $X$.

A version of Helly's property for CAT(0) cube complexes will be repeatedly used throughout this paper. We refer the reader to \cite{Ger_cubings} and \cite{Rol} for a proof. There is also a discussion regarding this property in \cite{CS}.

\begin{theorem}[Helly's Property]
	If $A_1, A_2, \dots A_n$ are convex, pairwise intersecting, subcomplexes (or alternatively are hyperplanes) of a CAT(0) cube complex, then $\bigcap_{i=1}^{n}A_i \neq \emptyset$. 
\end{theorem}

The next two lemmas are known by experts. We are not aware of an exact reference so we provide proofs. 

\begin{lemma} \label{lemma_convexity1}
	Let $Y$ be a convex subcomplex of a CAT(0) cube complex. Let $e$ be an edge adjacent to $Y$, and let $\hat{h}$ be the hyperplane dual to $e$. If $\hat{h}$ intersects $Y$, then $e \subset Y$.
\end{lemma}
\begin{proof}
	Let $u_1, u_2$ be the endpoints of $e$, with $u_1 \in Y$. Let $f$ be an edge in $\mathcal{C}(\hat{h}) \cap Y$. Let $v_1, v_2$ be the endpoints of $f$. By possibly relabeling, we may assume that $v_1$ and $u_1$ are both contained in the choice of halfspace, $h$.
	
	Let $\gamma$ be a geodesic from $u_2$ to $v_2$. As every hyperplane is dual to at most one edge of $\gamma$, $\gamma \subset \mathcal{C}^+(h^*)$. Let $p$ be the path obtained by concatenating the edge $e$ with the geodesic $\gamma$. It follows $p$ is also geodesic, as $\hat{h}$, and any other hyperplane, intersects $p$ exactly once. Furthermore, as $Y$ is convex, $p$ is contained in $Y$. We thus have that $u_2 \in p \subset Y$. Hence, the endpoints of $e$ are in $Y$. As $Y$ is convex, $e$ is also in $Y$.
\end{proof}

The following lemma states that given a shortest path between two convex subcomplexes, no hyperplane intersects both the path and one of the subcomplexes. The proof of this lemma is the only place in this article where we make use of disk diagrams. We refer the reader to \cite{Wise} for a background on disk diagrams.

\begin{lemma} \label{lemma_convexity2}
	Let $X$ be a CAT(0) cube complex, and let $A$ and $B$ be convex subcomplexes of $X$. Let $\gamma$ be a path from $A$ to $B$ that has minimal length out of all such possible paths. Let $\hat{h}$ be a hyperplane dual to an edge of $\gamma$. Then $\hat{h}$ is not dual to any edge in $A$ or in $B$.
\end{lemma}
\begin{proof}
	Suppose for a contradiction, that $\hat{h}$ is dual to both an edge, $e$, of $\gamma$ and an edge, $f$, of $A$. Furthermore, we choose $\hat{h}$ so that $e$ is closest to $A$ out of all such possible choices.
	
	Let $u_1$ and $u_2$ be the endpoints of $e$. Suppose $u_1$ is closer to $A$ than $u_2$. Let $h$ be the halfspace corresponding to $\hat{h}$ which contains $u_1$. Let $v_1$ and $v_2$ be the endpoints of $f$, labeled such that $v_1 \in h$. Let $a = \gamma \cap A$, and let $\gamma'$ be the subpath of $\gamma$ from $a$ to $u_1$.
	
	Let $D$ be a minimal area disk diagram with boundary $\gamma' \cup \eta \cup \zeta$, where $\eta$ is a geodesic from $u_1$ to $v_1$ and $\zeta$ is a geodesic from $v_1$ to $a$. Furthermore, suppose $D$ is minimal out of all possible choices for the geodesics $\eta$ and $\zeta$. As hyperplane carriers are convex, $\eta$ is contained in $\mathcal{C}^+(h)$. Similarly, $\zeta \subset A$ as $A$ is convex.
	
	We now apply \cite[Corollary 2.8]{Wise}. In that corollary we take the convex subcomplexes to be $\mathcal{C}^+(h)$ and $A$, and the paths between these subcomplexes to be $\gamma'$ and the length $0$ path, $v_1$. The conclusion of this corollary guarantees that two distinct dual curves in $D$ dual to $\eta$ do not intersect in $D$. Furthermore, by our choice of $\hat{h}$, no curve is dual to both $\gamma'$ and $\zeta$ (or else we could have chosen a hyperplane dual to an edge of $\gamma$ that is closer to $A$). Thus, every dual curve in $D$ has one end dual to $\eta$ and the other end dual to $\zeta \cup \gamma'$, and no two dual curves in $D$ intersect. Thus, we must have that $\eta = \zeta \cup \gamma'$. 
	
	Since $\mathcal{C}^+(h)$ is convex and $\zeta \cup \gamma'$ is a geodesic (as $\eta = \zeta \cup \gamma'$), it follows that $a \in \mathcal{C}^+(h)$. Let $f'$ be the edge dual to $\hat{h}$ that is adjacent to $a$. Let $b \in \mathcal{C}^+(h^*)$ be the other endpoint of $f'$. It follows by Lemma \ref{lemma_convexity1} that $b \in A$. 
	
	Let $\gamma'' \subset \mathcal{C}^+(h^*)$ be a geodesic from $b$ to $u_2$. Note that $|\gamma'| = |\gamma''|$ as $\mathcal{C}^+(h)$ and $\mathcal{C}^+(h^*)$ are isometric. However, if we now replace the subpath $e \cup \gamma'$ with the path $\gamma''$ in $\gamma$, we obtain a strictly smaller path from $A$ to $B$. This is a contradiction. 
	
	An argument showing $\hat{h}$ is not dual to an edge in $B$ is identical to the one given.
\end{proof}

\subsubsection{Euclidean and reducible complexes}
We say the CAT(0) cube complex $X$ is \textit{irreducible} if it is not the product of two CAT(0) cube complexes. Otherwise, we say that $X$ is \textit{reducible}.

A $n$--dimensional \text{flat} is an isometrically embedded copy of $\mathbb{E}^n$ (in the CAT(0) metric), where $n \ge 1$. A CAT(0) cube complex $X$ is \textit{Euclidean} if $X$ contains an $\text{Aut}(X)$ invariant flat. If $X$ is Euclidean, the \textit{Euclidean dimension} of $X$ is the largest $n$ for which $X$ contains an $\text{Aut}(X)$ invariant $n$--dimensional flat.

As different arguments are often required for reducible and Euclidean CAT(0) cube complexes, we treat such complexes separately in Section \ref{sec_euc_red}.

\subsection{The Roller and B(X) boundaries}
The definitions in this subsection follow those in \cite{NS}.

Let $X$ be a CAT(0) cube complex. An \textit{ultrafilter on $X$} is a collection, $\alpha$, of halfspaces in $X$ such that the following two conditions are satisfied:
\begin{enumerate}
	\item(choice condition) For every halfspace $h$ in $X$, either $h \in \alpha$ or $h^* \in \alpha$ (but not both).
	\item(consistency condition) Given halfspaces $h$ and $k$ in $X$ such that $h \in \alpha$ and $h \subset k$, it follows that $k \in \alpha$.
\end{enumerate}

The ultrafilter $\alpha$ satisfies the \textit{descending chain condition (DCC)} if every nested sequence of halfspaces in $\alpha$ contains a minimal element.

Let $\mathcal{U} = \mathcal{U}(X)$ be the set of ultrafilters on $X$. We may identify vertices in $X$ with ultrafilters in $\mathcal{U}$ satisfying the descending chain condition. Given a vertex, $v$, in $X$ we let $\alpha_v$ denote the ultrafilter corresponding to $v$. 

Let $\partial X$ denote the set of ultrafilters in $\mathcal{U}$ that do not satisfy the descending chain condition. We call $\partial X$ the \textit{Roller boundary of $X$}. 

Suppose that $X$ is additionally locally finite. We now define a metric on $\mathcal{U}$. Fix a base vertex $b \in X$. Given a hyperplane, it's distance from $b$ is given by:
\[d(\hat{h}, b) = | \big\{ \text{hyperplanes separating b from } \hat{h} \big\} | + 1 \]

Let $\alpha, \beta \in \mathcal{U}$ be ultrafilters. A hyperplane $\hat{h}$ separates $\alpha$ from $\beta$ if $h \in \alpha$ and $h^* \in \beta$ for some choice of halfspace $h$ for $\hat{h}$. We define the distance between two ultrafilters as:

\[d(\alpha, \beta) = \sup \Big\{ \frac{1}{ d(\hat{h}, b) } ~|~ \hat{h} \text{ separates } \alpha \text{ from } \beta \Big\} \]

Given a halfspace $h \subset X$, we define the open neighborhood, $\mathcal{U}_h$, to be the set of all ultrafilters that contain $h$. If $H$ is a finite subset of halfspaces in $X$, we define the open set $\mathcal{U}_H = \bigcap_{h \in H}\mathcal{U}_h$. Such open sets, together with $\mathcal{U}$, form a basis for the topology on $\mathcal{U}$.

Let $U \subset \mathcal{U}$ be an open set. Given a vertex $v \in X$, we say $v \in U$ if $U$ contains the ultrafilter, $\alpha_v$, corresponding to $v$. Given a hyperplane $\hat{h}$ in $X$, we say $\hat{h} \cap U \neq \emptyset$ if there exist adjacent vertices $v, v' \in U$ such that the edge between them is dual to $\hat{h}$. Similarly, we say $\hat{h} \cap U = \emptyset$ if no such pair of vertices exist. Finally, we say $\hat{h} \subset U$ if the endpoints of any edge dual to $\hat{h}$ are in $U$.

We say an ultrafilter $\alpha$ is \textit{non-terminating} if given any halfspace $h \in \alpha$, there exists a halfspace $k \in \alpha$ such that $k \subset h$. Denote by $\mathcal{U}_{NT} = \mathcal{U}_{NT}(X)$  the set of non-terminating ultrafilters in $\partial X$. The set \textit{$B(X)$} is defined to be the closure of $\mathcal{U}_{NT}$ in $\mathcal{U}$.

The following lemma follows from a straightforward application of the definitions.
\begin{lemma} \label{lemma_finite_subsets_in_limit}
	Let $X$ be a locally finite CAT(0) cube complex. A sequence of ultrafilters $\alpha_1, \alpha_2, \alpha_3 \dots$ in $X$ limit to $\alpha$ if and only if for every finite set of halfspaces $H \subset \alpha$, there exists an $N$ such that $H \subset \alpha_n$ for all $n \ge N$.
\end{lemma}

The next lemma records how a given halfspace of an ultrafilter must interact with an infinite chain of halfspaces in this ultrafilter.

\begin{lemma} \label{lemma_intersects_inf_many}
	Let $X$ be a CAT(0) cube complex. Let $h$ be a halfspace in the ultrafilter $\alpha \in \partial X$, and let $l_1 \supset l_2 \supset l_3 \dots$ be an infinite sequence of nested halfspaces in $\alpha$. Then there exists an $N$ such that either $\hat{h} \cap \hat{l}_n \neq \emptyset$ for all $n \ge N$ or $l_n \subset h$ for all $n \ge N$. In particular, if $h$ is minimal in $\alpha$, then $\hat{h} \cap \hat{l}_n \neq \emptyset$ for all $n > N$.
\end{lemma}
\begin{proof}
	If $l_N \subset h$ for some $N$, then $l_n \subset h$ for all $n \ge N$ and we are done. Thus we assume that for all $i \ge 1$, $h \subset l_i$ or $\hat{l}_i$ intersects $\hat{h}$. 
	
	Choose $N$ such that $\hat{l}_N$ does not intersect $\hat{h}$ and such that $d(\mathcal{C}(\hat{l}_N), \mathcal{C}(\hat{h})) \le d(\mathcal{C}(\hat{l}_i), \mathcal{C}(\hat{h}))$ for all $i$ where $\hat{l}_i \cap \hat{h} = \emptyset$. If no such $N$ exists, then $\hat{h}$ intersects $\hat{l}_i$ for all $i \ge 1$, and we are done.
	
	Let $n > N$. If $\hat{h} \subset l_n$, then $\mathcal{C}( \hat{h} )$ is strictly closer to $\mathcal{C}( \hat{l}_n )$ than to $\mathcal{C} (\hat{l}_N)$, contradicting our choice of $N$. On the other hand, it cannot be that $\hat{h} \subset l_n^*$, for $l_n \not\subset h$ by assumption, and $h \not\subset l_n^*$ as $\alpha$ satisfies the consistency condition. Hence, $\hat{l}_n$ must intersect $\hat{h}$ for all $n > N$.
\end{proof}

\section{Caged hyperplanes} \label{sec_caged_hyps}
We introduce the notion of caged hyperplanes, defined below, to give an equivalent condition for when $\partial X = B(X)$. We begin by defining a cage.

\begin{definition}[Cage] \label{def_small_sets}
	Let $\alpha$ be an ultrafilter on a CAT(0) cube complex $X$. A subset of halfspaces $K \subset \alpha$ is a \textit{cage in $\alpha$} if $\bigcap_{k \in K} \mathcal{C}^+(k) \neq \emptyset$.
\end{definition}

\begin{figure}[htp]
	\centering
	\begin{overpic}[scale=.7]{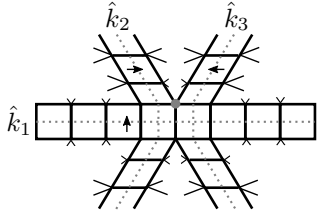}
		\put(-10,28){$\hat{k}_1$}	
		\put(25,65){$\hat{k}_2$}
		\put(67,65){$\hat{k}_3$}	
	\end{overpic}
	\caption{A cage, $K = \{k_1, k_2, k_3\}$. The dotted lines represent the hyperplanes $\hat{k}_1$,  $\hat{k}_2$ and $\hat{k}_3$. The arrows indicate the choice of halfspace for the corresponding hyperplane. The gray circle is a vertex in $\bigcap_{k \in K} \mathcal{C}^+(k)$.} \label{fig_cage}
\end{figure}

Note that any two halfspaces in a common cage are incomparable. For given halfspaces $h$ and $k$ such that $h \subset k$, then $\mathcal{C}^+(h) \cap \mathcal{C}^+(k) = \emptyset$.

Furthermore, we remark that a cage, $K$, in an ultrafilter on a locally finite CAT(0) cube complex, must contain finitely many halfspaces. This follows since given any vertex, $v \in \bigcap_{k \in K} \mathcal{C}^+(k)$, every hyperplane in $\hat{K}$ is dual to an edge adjacent to $v$.

We now define the property of having caged hyperplanes.
\begin{definition}[Caged Hyperplanes] \label{def_caged_hyps}
	A CAT(0) cube complex $X$ has \textit{caged hyperplanes} if given any ultrafilter $\alpha \in \partial X$ and a cage $K \subset \alpha$, there exists a hyperplane $\hat{h} \subset \bigcap_{k \in K}k$.  
\end{definition}

\begin{definition} \label{def_subset_containment}
Let $H$ and $H'$ each be a set of halfspaces. We say $H' \prec H$ if given any $h' \in H'$, there exists a halfspace $h \in H$ such that either $h' \subset h$ or $h' = h$. 
\end{definition}

By definition if $H' \prec H$ and $H'' \prec H'$, then $H'' \prec H$. We now show that given a finite subset of halfspaces, $H$, there exists a cage $K$ such that $K \prec H$. 
 
\begin{lemma} \label{lemma_cage_subset}
	Let $\alpha$ be an ultrafilter on a CAT(0) cube complex, and let $H \subset \alpha$ be a finite subset of halfspaces. Then there exists a cage $K \subset \alpha$ such that $K \prec H$.
\end{lemma}
\begin{proof}
	If there exist halfspaces $h, k \in H$ such that $k \subset h$, then set $H' = H \setminus h$. Clearly, we have that $H' \prec H$. By iteratively making such replacements, we obtain a subset of $H$ which does not contain a pair of comparable halfspaces. Thus, we may assume that $H$ does not contain comparable halfspaces. 
	
	We now construct the set of halfspaces, $K$, by making a series of replacements to $H$. Let $h, h' \in H$ be halfspaces such that $\mathcal{C}^+(h) \cap \mathcal{C}^+(h') = \emptyset$. Let $\gamma$ be a geodesic from $\mathcal{C}^+(h)$ to $\mathcal{C}^+(h')$. Let $\hat{k}$ be a hyperplane dual to an edge of $\gamma$. By Lemma \ref{lemma_convexity2}, the hyperplane $\hat{k}$ does not intersect $\mathcal{C}^+(h)$ and does not intersect $\mathcal{C}^+(h')$. Let $k$ be the choice of halfspace in $\alpha$ for $\hat{k}$. As $h$ and $h'$ are not comparable, either $k \subset h$ or $k \subset h'$. Without loss of generality, assume that $k \subset h$. We then form a new set of halfspaces $H_1$ by replacing $h$ by $k$. It follows that $H_1 \prec H$.
	
	We iteratively make such a replacements, when still possible, to obtain the sequence of subsets:
	\[H = H_0  \succ H_1 \succ H_2 \succ ... \succ H_n \]
	with $H_i \in \alpha$ for all $0 \le i \le n$.  Furthermore, as these replacements strictly decrease the sum of the distances between the carriers of hyperplanes in $\hat{H}$, this sequence is indeed finite ($n < \infty)$. Set $K = H_n$. Given $k, k' \in K$, we must have that $\mathcal{C}^+(h) \cap \mathcal{C}^+(h') \neq \emptyset$, or else another replacement could have been made. It follows that $\bigcap_{k \in K} \mathcal{C}^+k$ is nonempty by Helly's property. Thus, $K$ is a cage. 
\end{proof}

Given an ultrafilter on the Roller boundary, we show there are arbitrarily small neighborhoods, $\mathcal{U}_K$, containing this ultrafilter where $K$ a cage. Recall that $\mathcal{U}_K$ is the open set consisting of ultrafilters on $X$ that contain every halfspace in $K$.

\begin{proposition} \label{lemma_contains_cage_nbhd}
	Let $X$ be a locally finite CAT(0) cube complex. Given any $\alpha \in \partial X$ and open set $V \subset \mathcal{U}$ containing $\alpha$, there exists a cage, $K \subset \alpha$, such that $\alpha \in \mathcal{U}_K \subset V$.
\end{proposition}
\begin{proof}
	Choose $\epsilon > 0$ small enough so that the $\epsilon$--ball about $\alpha$ is contained in $V$.	Let $b \in X$ be the base vertex. Let $\hat{F}(R)$ be the set of hyperplanes in $X$ distance at most $R$ from $b$, and let $F(R)$ be the corresponding choices of halfspaces in $\alpha$ for the hyperplanes in $\hat{F}(R)$. As $X$ is locally finite, $|F(R)|$ is finite. By the definition of the metric on the ultrafilters $\mathcal{U}(X)$, there exists an $R > 0$, such that any ultrafilter containing $F(R)$ is $\epsilon$ close to $\alpha$. Fix such an $R$ and corresponding sets $F = F(R)$, $\hat{F} = \hat{F}(R)$.
	
	By Lemma \ref{lemma_cage_subset}, there exists a cage $K \subset \alpha$ such that $K \prec F$. An ultrafilter in $\mathcal{U}_K$ must contain every halfspace in $F$ by the consistency condition. Thus, every ultrafilter in $\mathcal{U}_K$ is distance at most $\epsilon$ from $\alpha$. Finally, $\mathcal{U}_K$ contains $\alpha$ as $K \subset \alpha$. This proves the claim.
\end{proof}

The next two propositions give some of the implications in Theorem \ref{thm_char}.
\begin{proposition} \label{prop_caged_equiv}
	Let $X$ be a locally finite CAT(0) cube complex. $X$ has caged hyperplanes if and only if every open set $V \subset \mathcal{U}$, satisfying $V \cap \partial X \neq \emptyset$, contains a hyperplane.
\end{proposition}
\begin{proof}
	Suppose first that $X$ has caged hyperplanes. Let $V$ be an open set containing an element $\alpha \in \partial X$. By Lemma \ref{lemma_contains_cage_nbhd}, there is a cage $K$ such that $\alpha \in \mathcal{U}_K \subset V$. As $X$ has caged hyperplanes, $\mathcal{U}_K$ contains a hyperplane. Thus, $V$ contains this hyperplane as well.
	
	Suppose, for the other direction, every open set contains a hyperplane. In particular for every cage, $K$, the open set, $\mathcal{U}_K$, contains a hyperplane. Thus, $X$ has caged hyperplanes.
\end{proof}

\begin{proposition} \label{prop_caged_implies_bx}
	Let $X$ be a cocompact, essential, locally finite CAT(0) cube complex. If $X$ has caged hyperplanes, then $\partial X = B(X)$
\end{proposition}
\begin{proof}
	Let $\alpha \in \partial X$ and $V \subset \mathcal{U}$ a neighborhood of $\alpha$. By Proposition \ref{lemma_contains_cage_nbhd}, there exists a cage $K$ such that the corresponding open set, $\mathcal{U}_K$, satisfies $\alpha \in \mathcal{U}_K \subset V$. By assumption, there is a hyperplane $\hat{h} \subset \mathcal{U}_K$. Let $h$ be the choice of halfspace for $\hat{h}$ such that $\alpha \in \mathcal{U}_h$. In particular, $h \subset k$ for every $k \in K$.
	
	By \cite[Theorem 3.1]{NS} and Remark 3.2 after, there exists a non-terminating ultrafilter $\beta$ that contains $h$. By the consistency condition, $k \in \beta$ for every $k \in K$. Thus, $\beta \in \mathcal{U}_K$. Hence, there exist non-terminating ultrafilters arbitrarily close to $\alpha$. This shows $\alpha$ is in the closure of $\mathcal{U}_{NT}$. Thus, $\partial X = B(X)$.
\end{proof}

\section{Straight Links} \label{sec_straight_links}

We define when a CAT(0) cube complex has straight links. Every CAT(0) cube complex with extendable CAT(0) geodesics has straight links. This follows from \cite[Proposition II.5.10]{BH}. It is straightforward to find examples where the converse is not true. Thus, the straight links assumption is a natural one to make. Throughout this section we prove basic results regarding CAT(0) cube complexes with straight links.

\begin{definition}[Straight links] \label{def_ecg}
	A CAT(0) cube complex $X$ has \textit{straight links}, if given any vertex $v \in X$ and any edge $e$ adjacent to $v$, there exists an edge $f$ adjacent to $v$ such that the hyperplane dual to $e$ does not intersect the hyperplane dual to $f$.
\end{definition}

\begin{lemma} \label{lemma_straight_links_implies_essential}
	Every CAT(0) cube complex with straight links is essential.
\end{lemma}
\begin{proof}
	Let $X$ be a CAT(0) cube complex with straight links. Let $\hat{h}$ be a hyperplane in $X$, and let $h$ be a choice of halfspace for $\hat{h}$. Let $e$ be an edge dual to $\hat{h}$, and let $v$ be the endpoint of $e$ contained in $h$. As $X$ has straight links, there exists an infinite sequence of edges $e = e_0, e_1, e_2, \dots$ such that the hyperplanes dual to these edges do not pairwise intersect and such that $e \cap e_1 = v$. Thus for each $i \ge 1$, $e_i \subset h$ and $d(e_i, \mathcal{C}(\hat{h})) = i - 1$. It follows that $X$ is essential.
\end{proof}

A version of the following definition is given in \cite{NS}.
\begin{definition}[Deep] \label{def_deepness}
	Let $h$ be a halfspace in a CAT(0) cube complex $X$ and let $Z$ be a subcomplex of $X$. We say that \textit{$Z$ is deep in $h$} if $Z \cap h$ is not contained in a finite neighborhood of $\hat{h}$. Otherwise, we say $Z$ is \textit{shallow} in $h$.
\end{definition}

The next lemma gives properties of certain cages in a CAT(0) cube complex with straight links.

\begin{lemma}  \label{lemma_straight_link_cages}
	Let $X$ be a CAT(0) cube complex with straight links. Let $K$ be a cage in $X$. Suppose for some $k_0 \in K$, $\hat{k}_0 \cap \hat{k} \neq \emptyset$ for all $k \in K$. Then 
	\begin{enumerate}
		\item \label{lemma_straight_link_cages1} There exists an ultrafilter $\alpha \in \partial X$ such that $K \subset \alpha$.	
		\item \label{lemma_straight_link_cages2} If $\bigcap_{k \in K}k$ does not contain a hyperplane, then $\hat{k}$ is deep in $k_0$ for some $k \in K$.
	\end{enumerate}
\end{lemma}
\begin{proof}
	We first prove \ref{lemma_straight_link_cages1}. By the definition of a cage, there exists a vertex $v \in \bigcap_{k \in K} \mathcal{C}^+(k)$. Let $e$ be the edge dual to $\hat{k}_0$ that is adjacent to $v$. By the straight links condition, there is a geodesic ray, $\gamma$, that is the concatenation of edges $e_1, e_2, e_3 \dots$ such that $e \cap e_1 = v$ and such that the hyperplane dual to $e_i$ does not intersect the hyperplane dual to $e_j$ for all $0 \le i < j$. Furthermore, $\hat{k}_0$ does not intersect $\gamma$.
	
	We claim that $\hat{k} \cap \gamma = \emptyset$ for all $k \in K$. For suppose otherwise that $\hat{k} \cap \gamma \neq \emptyset$ for some $\hat{k} \in \hat{K}$. By our choice of $\gamma$, $\hat{k} \neq \hat{k}_0$. Let $e_i$ be the edge of $\gamma$ that is dual to $\hat{k}$. 
	
	If $i = 1$, then $e_1$ and $e$ are adjacent edges dual to intersecting hyperplanes ($\hat{k} \cap \hat{k}_0 \neq \emptyset$ by assumption). This is not possible by our choice of $\gamma$. On the other hand, suppose that $i > 1$. Let $\gamma'$ be the subpath of $\gamma$ from $v$ to $e_i$. By convexity, $\gamma' \subset \mathcal{C}^+(k)$. However, $e_{i-1} \in \mathcal{C}^+(k)$ and $e_i$ is dual to $\hat{k}$. This means that the hyperplane dual to $e_{i-1}$ intersects the hyperplane dual to $e_i$. This again contradicts our choice of $\gamma$. Thus, $\hat{k} \cap \gamma = \emptyset$ for all $k \in K$. 
	
	Let $\hat{l}_1,~ \hat{l}_2, \dots$ be the set of hyperplanes dual to the edges $e_1,~ e_2, \dots$. Choose halfspaces $l_1 \supset l_2 \supset l_3 \dots$ such that $v \in l_i^*$ for all $i \ge 1$. Note that $\hat{l}_i \not\subset k^*$ for any $k \in K$, since $\hat{k}$ does not intersect $\gamma$ and $v \in k$ for any $k \in K$. It follows that the set of halfspaces $H = K \cup \bigcup_{i=1}^\infty l_i$ satisfies the consistency condition. By \cite[Section 3.3]{Rol} there is some ultrafilter, $\alpha$ containing $H$. Furthermore, $\alpha \in \partial X$ as it contains an infinite descending chain of halfspaces. Thus, \ref{lemma_straight_link_cages1} follows.
	
	To prove \ref{lemma_straight_link_cages2}, note that if $\hat{k} \cap k_0$ were shallow for every $k \in K$, then for large enough $N$, we would have that $l_N \subset k$ for all $k \in K$. However, this is not possible if $\bigcap_{k \in K}k$ does not contain a hyperplane.
\end{proof}

The following lemma, which is a direct consequence of \cite[Proposition 2.7]{Hagen-Susse}, shows that if $G$ acts cocompactly on $X$, then a subgroup of $G$ acts cocompactly on the intersection of a finite set of hyperplane carriers. We note that the straight links hypothesis is not needed in this lemma. 

\begin{lemma} \label{lemma_cocompact_action_on_cores}
	Let $G$ act cocompactly on a locally finite CAT(0) cube complex $X$, and let $H$ be a set of halfspaces in $X$ such that $Y = \bigcap_{h \in H}\mathcal{C}^+(h)$ is non-empty. Let $G'$ be the subgroup of $G$ that stabilizes $Y$ and stabilizes each halfspace $h \in H$. Then $G'$ acts cocompactly on $Y$.
\end{lemma}
\begin{proof}
	The subcomplex $Y$ is an element of a factor system (we refer to \cite{Hagen-Susse} for a definition) on $X$. By \cite[Proposition 2.7]{Hagen-Susse}, the stabilizer $G''$ of $Y$ acts cocompactly on $Y$. As $X$ is locally finite, $H$ is finite, and $G''$ permutes the halfspaces in $H$. Thus, $G'$ is a finite index subgroup of $G''$. Consequently, $G'$ acts cocompactly on $Y$.
\end{proof}

\begin{definition}[Sector]
	Let $\mathcal{S}$ be a finite collection of halfspaces in a CAT(0) cube complex such that $\hat{s} \cap \hat{s}' \neq \emptyset$ for all $\hat{s}, \hat{s}' \in \hat{\mathcal{S}}$. We call the intersection $\bigcap_{s \in \mathcal{S}}s$ a \textit{sector}.
\end{definition}

The lemma below shows that if a CAT(0) cube complex satisfies certain assumptions, then every sector in the cube complex contains a hyperplane.

\begin{lemma} \label{lemma_straight_link_sectors}
	Let $X$ be a cocompact, locally finite, irreducible, non-Euclidean CAT(0) cube complex with straight links, then every sector contains a hyperplane.
\end{lemma}
\begin{proof}
	As $X$ has straight links, $X$ is essential by Lemma \ref{lemma_straight_links_implies_essential}.	Let $\mathcal{S} = \{s_1, \dots, s_n\}$ be a set of halfspaces such that $s_1 \cap \dots \cap s_n$ is a sector.	We prove the claim by induction on $n$. When $n = 1$, the claim follows as $X$ is essential. When $n=2$, $\hat{s}_2 \cap s_1$ is deep by Lemma \ref{lemma_straight_link_cages} \ref{lemma_straight_link_cages2} (straight links hypothesis is used here). As $X$ is cocompact, essential, irreducible and non-Euclidean, by \cite[Lemma 5.5]{NS} the sector $s_1 \cap s_2$ contains a hyperplane, and the claim follows in this case as well.
	
	Assume now that $n>2$ and that the claim is true for any sector with less than $n$ halfspaces. In particular, the sector $s_2 \cap s_3 \cap  \dots \cap s_n$ contains a hyperplane $\hat{h}$. If $\hat{h} \subset s_1$, then we are done.
	
	Suppose $\hat{h}$ intersects $\hat{s}_1$. Let $h$ denote the choice of halfspace for $\hat{h}$ such that $h \subset s_i$ for all $2 \le i \le n$. By induction, the sector $h \cap s_1$ contains a hyperplane $\hat{k}$. We are then done in this case as $\hat{k} \subset s_1 \cap \dots \cap s_n$.
	
	For the final case, suppose $\hat{h} \subset s_1^*$. Let $Z = \bigcap_{i=2}^{n}\mathcal{C}^+(s_i)$, which is not empty by Helly's property. The proof of this case breaks down to two subcases depending on whether $Z$ is deep in $s_1$.
	
	Suppose first that $Z$ is not deep in $s_1$. Choose a geodesic $\gamma$ from $\mathcal{C}^+(s_1)$ to some vertex $v \in Z$ such that $|\gamma|$ is maximal. As $X$ has straight links, there exists an edge $e$ adjacent to $v$ such that the hyperplane, $\hat{k}$, dual to $e$ does not intersect $\gamma$. Note that the concatenation $\gamma \cup e$ is a geodesic as well. 
	
	We claim $\hat{k}$ does not intersect some hyperplane, $\hat{s}_r \in \hat{\mathcal{S}}$, where $2 \le r \le n$. For, suppose otherwise, that $\hat{k} \cap \hat{s}_i \neq \emptyset$ for all $2 \le i \le n$, then by Helly's property $\hat{k} \cap Z \neq \emptyset$. As $Z$ is convex, this implies that $e \subset Z$ by Lemma \ref{lemma_convexity1}. However, as $\gamma \cup e$ is a geodesic of longer length than $\gamma$, this contradicts our choice of $\gamma$.
	
	Let $\mathcal{S}_{\perp}$ be the set of halfspaces in $\mathcal{S}$ whose corresponding hyperplanes intersect $\hat{k}$, and let $\mathcal{S}_{\parallel}$ be the set of halfspaces in $\mathcal{S}$ whose corresponding hyperplanes do not intersect $\hat{k}$. It follows that $s_r \in \mathcal{S}_{\parallel}$ by what we have shown. Furthermore, $s_1 \in \mathcal{S}_{\parallel}$ by Lemma \ref{lemma_convexity2}. Thus, $|\mathcal{S}_{\perp}| \le n-2$. 
	
	Let $k$ be the choice of halfspace for $\hat{k}$ such that $k \subset s$ for all $s \in \mathcal{S}_{\parallel}$. By the induction hypothesis, the sector $k \cap (\bigcap_{s \in \mathcal{S}_{\perp}}{s})$ contains a hyperplane $\hat{m}$. Thus, $\hat{m} \subset s_1 \cap \dots \cap s_n$.
	
	On the other hand, suppose $Z$ is deep. Let $G'$ be the subgroup of $G$ that stabilizes $Z$ and stabilizes the halfspace $s_i$ for each $2 \le i \le n$. As $G$ acts cocompactly and $X$ is locally finite, by Lemma \ref{lemma_cocompact_action_on_cores}, $G'$ acts cocompactly on $Z$. Thus, there exists some element $g \in G'$ such that 
	\[g \hat{h} \subset s_2 \cap \dots \cap s_n \text{ and } g \hat{h} \cap s_1 \neq \emptyset\] 
	If $g \hat{h}$ is contained in $s_1$ then we are done. Otherwise, $g \hat{h}$ intersects $\hat{s}_1$ and we are done by applying the previous case.
\end{proof}

\section{Tight Cages} \label{sec_tight_cages}

In this section, we introduce tight cages. Under the extra assumption of straight links, we will show that  $\partial X \neq B(X)$ if and only if $X$ contains a tight cage. In Section \ref{sec_2d_case}, the structure of tight cages is used to deduce Theorem \ref{thm_2_dim_ccc} on $2$--dimensional CAT(0) cube complexes.

\begin{figure}[htp]
	\centering
	\begin{overpic}[scale=.5]{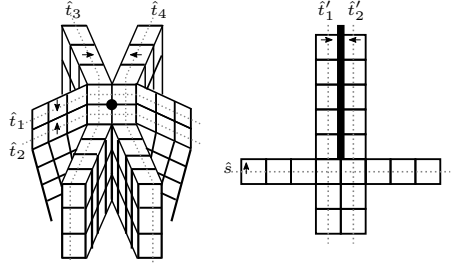}
		\put(-3,32){\tiny $\hat{t}_1$}	
		\put(-3,25){\tiny $\hat{t}_2$}	
		\put(10,58){\tiny $\hat{t}_3$}	
		\put(30,58){\tiny $\hat{t}_4$}
		\put(69,58){\tiny $\hat{t}_1'$}	
		\put(76,58){\tiny $\hat{t}_2'$}	
		\put(47,21){\tiny $\hat{s}$}	
	\end{overpic}
	\caption{Sections of tight cages, $(\mathcal{S} = \{ \emptyset \}, \mathcal{T} = \{t_1, t_2, t_3, t_4\})$ on the left and $(\mathcal{S}' = \{ s \}, \mathcal{T}' = \{t_1', t_2'\})$ on the right, are shown. The dotted lines represent hyperplanes, and the arrows represent the choice of halfspace. The bold circle on the left and the bold line on the right represent segments in the cores of the tight cages. 
	} \label{fig_tight_cage}
\end{figure}

\begin{definition}[Tight Cage] \label{def_tight_cage}
	A \textit{tight cage} is a pair, $(\mathcal{S}, \mathcal{T}$), where $\mathcal{S}$ and $\mathcal{T}$ are each a finite collection of halfspaces satisfying:
	\begin{enumerate}
		\item \label{def_tight cage_large_T} $|\mathcal{T}| \ge 2$
		
		\item \label{def_tight_cage_intersections} For all $s \in \mathcal{S}$ and all $h  \in \mathcal{S} \cup \mathcal{T}$, $\hat{s} \cap \hat{h} \neq \emptyset$. In particular, $\bigcap_{s \in S}s$ is a sector.
		 
		\item \label{def_tight cage_unbounded_core}  The set 
		\[Y = \bigcap_{t \in \mathcal{T}} \mathcal{C}^+(t) \cap \bigcap_{s \in \mathcal{S}} s\]
		 is unbounded. We call $Y$ the \textit{core} of the tight cage.
		 
		\item \label{def_tight cage_core_prop}  Given a hyperplane $\hat{h}$ that is dual to an edge adjacent to the core $Y$, either $\hat{h} \in \hat{\mathcal{S}} \cup \hat{\mathcal{T}}$ or $\hat{h}$ intersects every hyperplane in $\hat{\mathcal{T}}$.
	\end{enumerate} 
	We refer to the four conditions above as the \textit{tight cage conditions}.
\end{definition}

\begin{remark}
	In the definition of a tight cage, $\mathcal{S}$ is allowed to be empty.
\end{remark}
\begin{remark}
	The core of a tight cage is always convex as it is the intersection of convex sets.
\end{remark}

We first check that every tight cage is indeed a cage and that this cage does not contain a hyperplane.

\begin{proposition} \label{prop_tight_cages_are_cages}
	Let $(\mathcal{S}, \mathcal{T})$ be a tight cage, then $K = \mathcal{S} \cup \mathcal{T}$ is a cage that does not contain a hyperplane.
\end{proposition}
\begin{proof}
	By tight cage condition \ref{def_tight_cage_intersections}, tight cage condition \ref{def_tight cage_unbounded_core} and Helly's property, it follows that the intersection
	\[ \bigcap_{t \in \mathcal{T}} \mathcal{C}^+(t) \cap \bigcap_{s \in \mathcal{S}} \mathcal{C}^+(s) \]
	is non-empty. Thus $K$ is a cage.
	
	Let $Y$ be the core of the tight cage $(\mathcal{S}, \mathcal{T})$. Suppose for a contradiction, that there exists a hyperplane $\hat{h} \subset \bigcap_{k \in K}k$. Such a hyperplane cannot intersect $Y$ by definition. If $\hat{h}$ is dual to an edge that is adjacent to a vertex of $Y$, then by tight cage condition \ref{def_tight cage_core_prop}, it follows that $\hat{h} \in \hat{K}$. This is also not possible. We may thus assume that $d(\mathcal{C}(\hat{h}), Y) >0$.
	
	Let $\gamma$ be a geodesic from $\mathcal{C}(\hat{h})$ to $Y$. By Lemma \ref{lemma_convexity2}, the hyperplane, $\hat{k}$, dual to the last edge of $\gamma$ does not intersect $Y$ and does not intersect $\hat{h}$. Let $k$ be the choice of halfspace for $\hat{k}$ so that $k$ contains $Y$. 
	
	As $\hat{k}$ is dual to an edge adjacent to $Y$ and $\hat{k}$ does not intersect $Y$, it follows that $k \in K$ by tight cage condition \ref{def_tight cage_core_prop}. However, this implies $\hat{h} \not\subset k$, a contradiction.
\end{proof}

Let $\Gamma$ be a graph. Recall that $\Gamma$ is a \textit{join} if there is a nontrivial decomposition of the vertices of $\Gamma$, $V(\Gamma) = V_1 \cup V_2$, such that every vertex in $V_1$ is adjacent to every vertex in $V_2$. If $\Gamma$ is a flag simplicial complex (e.g. $\Gamma$ is the link of a vertex of a CAT(0) cube complex), then we say $\Gamma$ is a join if its $1$--skeleton is a join.

A consequence of $X$ containing a tight cage is that $X$ contains a convex, unbounded subset $Y$ (namely the core) such that the link, taken in $X$, of every vertex in $Y$ is a join. This is proven in the proposition below. This proposition is particularly useful as it is often easy to deduce that $X$ does not contain such a subset and, consequently, does not contain a tight cage.

\begin{proposition} \label{prop_core_links}
	Let $(\mathcal{S}, \mathcal{T})$ be a tight cage in a CAT(0) cube complex $X$,  and let $Y$ be the core of the tight cage. Then the link, taken in $X$, of any vertex in $Y$ is a join.
\end{proposition}
\begin{proof}
	Let $v \in Y$. Let $\hat{H}$ be the set of hyperplanes which are dual to an edge adjacent to $v$ and which intersect $Y$. As $Y$ is connected and unbounded (tight cage condition \ref{def_tight cage_unbounded_core}), the set $\hat{H}$ is not empty. Let $\hat{\mathcal{S}}'$ be the (possibly empty) subset of hyperplanes in $\hat{\mathcal{S}}$ that are dual to an edge adjacent to $v$. 
	
	As $Y$ is contained in the intersection of positive carriers of hyperplanes in $\mathcal{T}$, it follows that every hyperplane in $\mathcal{T}$ is dual to an edge adjacent to $v$. Furthermore, by definition of a tight cage, every hyperplane dual to an edge adjacent to $v \in Y$ is in the set $\hat{\mathcal{S}'} \cup \hat{\mathcal{T}} \cup \hat{H}$. Additionally, by tight cage conditions \ref{def_tight cage_core_prop} and \ref{def_tight_cage_intersections}, all hyperplanes in $\hat{H} \cup \hat{S}'$ intersect all hyperplanes in $\hat{\mathcal{T}}$, and by tight cage condition \ref{def_tight cage_large_T} the set $\hat{\mathcal{T}}$ is not empty.
	
	Let $\Gamma$ be the link of $v$. Every vertex of $\Gamma$ is naturally contained in a distinct hyperplane of $X$. Let $V_1$ be the set of vertices in $V(\Gamma)$ contained in $\hat{H} \cup \hat{S}'$, and let $V_2$ be the set of vertices in $\Gamma$ contained in $\hat{\mathcal{T}}$. Note that two vertices in $\Gamma$ are adjacent if the corresponding hyperplanes they are contained in intersect.	It follows that $V_1 \cup V_2$ is a join decomposition of $\Gamma$.
\end{proof}

We are now ready to prove the main result of this section. 

\begin{proposition} \label{prop_tight cage}
	Let $X$ be a cocompact, locally finite, irreducible, non-Euclidean, CAT(0) cube complex with straight links. Then $X$ contains a tight cage if and only if $X$ does not have caged hyperplanes.
\end{proposition}
\begin{proof}
	For one direction, suppose that $X$ contains a tight cage $(\mathcal{S}, \mathcal{T})$. By Proposition \ref{prop_tight_cages_are_cages}, $K = \mathcal{S} \cup \mathcal{T}$ is a cage and $\bigcap_{k \in K}k$ does not contain a hyperplane. If $\mathcal{S} \neq \emptyset$, then given $s \in \mathcal{S}$, $\hat{s}$ intersects every hyperplane in $\hat{K}$. Consequently, as $X$ has straight links, by Lemma \ref{lemma_straight_link_cages}(\ref{lemma_straight_link_cages1}) there exists an ultrafilter $\alpha \in \partial X$ such that $K \subset \alpha$. Thus, in this case $X$ does not have caged hyperplanes. 
	
	On the other hand, suppose that $\mathcal{S} = \emptyset$.  Let $Y$ be the core of $(\mathcal{S}, \mathcal{T})$. By tight cage condition \ref{def_tight cage_unbounded_core}, $Y$ is unbounded and connected. In particular, there exists a hyperplane $\hat{h}$ that intersects $Y$. Consequently, $\hat{h}$ intersects every hyperplane in $\mathcal{T}$. Let $h$ be any choice of halfspace for $\hat{h}$. It follows that $K' = \mathcal{T} \cup h$ is a cage and $\bigcap_{k \in K'}k$ does not contain a hyperplane. By Lemma \ref{lemma_straight_link_cages}(\ref{lemma_straight_link_cages1}), there exists an ultrafilter in $\partial X$ that contains $K'$. Thus, $X$ does not have caged hyperplanes in this case as well. This proves one direction of the proposition.
	
	We now assume that $X$ does not have caged hyperplanes, and we show $X$ contains a tight cage.	We begin by defining a complexity function, $\Theta$, that assigns an integer to every finite set of hyperplanes. Namely, given a finite set of hyperplanes, $H$, let $\Theta(H)$ denote the number of unordered pairs of non-intersecting hyperplanes in $\hat{H}$:
	\[\Theta(H) = \Big|\Big\{ \{\hat{h}, \hat{k}\} \subset \hat{H} ~|~ \hat{h} \cap \hat{k} = \emptyset \Big\} \Big| \]
	
	As $X$ does not have caged hyperplanes, there is an ultrafilter $\alpha \in \partial X$ and a cage $K$ in $\alpha$ such that $\bigcap_{k \in K}k$ does not contain a hyperplane. Furthermore, we choose $K$ and $\alpha$ so that $\Theta(K)$ is minimal out of all such possible choices.
	
	Note that $\Theta(K) = 0$ if and only if the intersection of the halfspaces in $K$ forms a sector. By Lemma \ref{lemma_straight_link_sectors} (note that this lemma requires every hypothesis that we have placed on $X$) any sector in $X$ contains a hyperplane. Consequently $K$ cannot be a sector, and $\Theta(K) \ge 1$.
	
	Let $\mathcal{S}$ be halfspaces in $K$ whose corresponding hyperplane intersects every hyperplane in $\hat{K}$. Namely,
	\[\mathcal{S} = \{s \in K \mid \hat{s} \cap \hat{k} \neq \emptyset, ~\forall \hat{k} \in \hat{K} \} \]
	In particular, the intersection of halfspaces in $\mathcal{S}$ is a sector. Let $\mathcal{T} = K \setminus \mathcal{S}$. Note that $|\mathcal{T}| \ge 2$ since $\Theta(K) \ge 1$. Set 
	\[Y = \bigcap_{t \in  \mathcal{T}} \mathcal{C}^+(t) \cap \bigcap_{s \in \mathcal{S}} s\] 	
		
	We will show that $(\mathcal{S}, \mathcal{T})$ is a tight cage with $Y$ the corresponding core. We first prove the following subclaim:
	
	\begin{subclaim}
	 Let $\hat{h}$ be a hyperplane in $X$ such that for every $k \in K$, either $\hat{h} \subset k$ or $\hat{h} \cap \hat{k} \neq \emptyset$. Then $\hat{h} \cap \hat{t} \neq \emptyset$ for all $\hat{t} \in \hat{\mathcal{T}}$.	
	\end{subclaim}
	\begin{proof}		
	Suppose, for a contradiction, that $\hat{h}$ does not intersect the hyperplane, $\hat{t}_0 \in \hat{\mathcal{T}}$. In particular, $\hat{h} \cap Y = \emptyset$. Let $h$ be the choice of halfspace for $\hat{h}$ so that $Y \subset h^*$. Let $\hat{\mathcal{S}}'$ be the subset of hyperplanes in $\hat{\mathcal{S}}$ that intersect $\hat{h}$, and let $\hat{\mathcal{T}}'$ be the subset of hyperplanes of $\hat{\mathcal{T}}$ that intersect $\hat{h}$. Let $\mathcal{S}'$ and $\mathcal{T}'$ be the corresponding set of halfspace subsets of $\mathcal{S}$ and $\mathcal{T}$.
	
	Set $K' = h \cup \mathcal{T}' \cup \mathcal{S}'$. By Helly's property, $K'$ is a cage. We claim two facts: $\bigcap_{k \in K'}k$ does not contain a hyperplane and $\Theta(K') < \Theta(K)$.
	
	To see the first claim, note that by our hypothesis on $\hat{h}$, it follows that $h \subset k$ for every $k \in K \setminus (\mathcal{T}'	\cup \mathcal{S}')$. Thus $\bigcap_{k \in K'}k \subset \bigcap_{k \in K}k$ (i.e., $K' \prec K$ as in Definition \ref{def_subset_containment}). In particular, $\bigcap_{k \in K'}k$ cannot contain a hyperplane as $\bigcap_{k \in K}k$ does not contain a hyperplane. 
	
	We now show the second claim regarding $K'$. By our assumption, there exists a halfspace $t_0 \in \mathcal{T}$ such that $t_0 \notin K'$. By construction of $\mathcal{T}$, there exists a hyperplane in $\hat{\mathcal{T}}$ which does not intersect $\hat{t}_0$. Thus, $\Theta(\mathcal{T}' \cup \mathcal{S}') < \Theta(K)$. Furthermore, as $K' = h \cup \mathcal{T}' \cup \mathcal{S}'$ and $\hat{h}$ intersects every hyperplane in $K$, it follows that $\Theta(K') = \Theta(\mathcal{T}' \cup \mathcal{S}') < \Theta(K)$. This establishes the second claim.
	
	As $\hat{h}$ intersects every hyperplane in $\hat{K'}$, by Lemma \ref{lemma_straight_link_cages}(\ref{lemma_straight_link_cages1}), there exists an ultrafilter $\beta \in \partial X$, such that $K' \subset \beta$. However, $K'$ is a cage in $\beta \in \partial X$ such that $\bigcap_{k \in K'}k$ does not contain a hyperplane and $K'$ has strictly smaller complexity than $K$ ($\Theta(K') < \Theta(K)$). This contradicts the minimality of our choice of $K$ and $\alpha$. Thus, the subclaim follows.
	\end{proof}
	We now check that $(\mathcal{S}, \mathcal{T})$ satisfies the tight cage conditions. Since $|\mathcal{T}| \ge 2$, tight cage condition \ref{def_tight cage_large_T} is satisfied. Tight cage condition \ref{def_tight_cage_intersections} is satisfied by our choice of $\mathcal{S}$.
	
	We now check tight cage condition \ref{def_tight cage_core_prop}. Let $\hat{h}$ be a hyperplane, that is not in $\hat{K} = \hat{\mathcal{S}} \cup \hat{\mathcal{T}}$, and is dual to an edge adjacent to $Y$. It follows that for every $k \in K$, either $\hat{h} \subset k$ or $\hat{h} \cap \hat{k} \neq \emptyset$. By the sub-claim, $\hat{h}$ must intersect $\hat{t}$ for every $\hat{t} \in \mathcal{T}$.
	
	All that is left to show is that $Y$ is unbounded, i.e. tight cage condition \ref{def_tight cage_unbounded_core}. Let $l_1 \supset l_2 \supset l_3 \dots $ be a sequence of nested halfspaces in $\alpha$. Such a sequence must exist since $\alpha \in \partial X$. By Lemma \ref{lemma_intersects_inf_many} we may assume, by possibly passing to a subsequence, that for each $k \in K$ and for each $i \ge 1$ either $\hat{l}_i \subset k$ or $\hat{l}_i \cap \hat{k} \neq \emptyset$. It follows by the sub-claim that $\hat{l}_i \cap \hat{t} \neq \emptyset$ for all $i \ge 1$ and $\hat{t} \in \hat{\mathcal{T}}$. By Helly's property, $\mathcal{C}(l_i) \cap Y \neq \emptyset$ for all $i \ge 1$. In particular, $Y$ is unbounded. Thus, $(\mathcal{S}, \mathcal{T})$ is indeed a tight cage.
\end{proof}

\section{Euclidean and Reducible Cases} \label{sec_euc_red}

In this section, we focus on CAT(0) cube complexes that are either Euclidean or reducible. We begin with the product case below. In this setting, $B(X)$ is never equal to $\partial X$.

\begin{lemma} \label{lemma_reducible_case}
	Let $X = X_1 \times X_2$ be a locally finite, reducible CAT(0) cube complex such that $X_1$ contains an edge and $X_2$ is cocompact and essential, then $\partial X \neq B(X)$, $X$ does not have caged hyperplanes and $X$ contains a tight cage.
\end{lemma}
\begin{proof}
	As $X$ is a product, the hyperplanes in $X$ form two disjoint sets, $\hat{H}_1$ and $\hat{H}_2$, and every hyperplane in $\hat{H}_1$ intersects every hyperplane in $\hat{H}_2$ \cite[Lemma 2.5]{CS}. As both $X_1$ and $X_2$ contain an edge, $\hat{H}_1$ and $\hat{H}_2$ are not empty. 
	
	Fix a vertex $v \in X_1$ and let $\alpha_1 = \alpha_v$ be the ultrafilter on $X_1$ whose halfspaces each contain $v$. The ultrafilter $\alpha_1$ satisfies the descending chain condition (DCC). Let $M$ be the set of minimal halfspaces in $\alpha_1$. The halfspaces in $M$ are exactly the halfspaces, $h$, in $X_1$ such that $v \in \mathcal{C}^+(h)$. In particular, the set $|M|$ is finite (as $X_1$ is locally finite) and is non-empty (as $X_1$ is connected and contains an edge).
	
	As $X_2$ is cocompact and essential, by \cite[Theorem 3.1]{NS} there exists a non-terminating ultrafilter $\alpha_2$ on $X_2$. It follows that $\alpha = \alpha_1 \times \alpha_2$ is an ultrafilter on $X$ that does not satisfy the DCC. In particular, $\alpha \in \partial X$. 
	
	Every halfspace, $h$, in $X_i$, $i = 1, 2$, can be naturally associated to a halfspace, $\pi_i^{-1}(h)$, in $X$ where $\pi_i: X \to X_i$ is the natural projection. Let $\bar{M} = \{\pi_1^{-1}(m) ~|~ m \in M\}$ be the set of halfspaces in $X$ corresponding to the halfspaces in $M$. Similarly, let $\bar{\alpha}_i = \{\pi_i^{-1}(h) ~|~ h \in \alpha_i \}$ for $i = 1, 2$.
	
	As every hyperplane corresponding to a halfspace in $\bar{\alpha}_1$ intersects every hyperplane corresponding to a halfspace in $\bar{\alpha_2}$, the halfspaces $\bar{M}$ are minimal in $\alpha$, and it is also straightforward to check that $\{\mathcal{S} = \emptyset, \mathcal{T} = \bar{M} \}$ is a tight cage in $X$. Furthermore, by Proposition \ref{prop_tight_cages_are_cages} $\bar{M}$ is a cage and $\bigcap_{m \in \bar{M}}m$ does not contain a hyperplane. Thus, as $\bar{M}$ is contained in the ultrafilter $\alpha \in \partial X$, $X$ does not have caged hyperplanes.
	
	We now prove that $\partial X \neq B(X)$ by showing that the ultrafilter $\alpha$, as defined above, cannot be the limit of a sequence of non-terminating ultrafilters, and consequently $\alpha \in \partial X \setminus B(X)$. 
	Assume, for a contradiction, that $\alpha$ is the limit of a sequence of non-terminating ultrafilters. By Lemma \ref{lemma_finite_subsets_in_limit}, there must exist a non-terminating ultrafilter $\beta$ that contains every halfspace in $\bar{M}$. 
	
	Fix a halfspace $\bar{m} \in \bar{M}$. As $\beta$ is non-terminating, there exists a halfspace $\bar{k} \in \beta$ such that $\bar{k} \subset \bar{m}$. In particular, $k = \pi_1(\bar{k})$ is a halfspace in $X_1$. Let $m = \pi_1(\bar{m})$. It follows that $v \in \mathcal{C}^+(m)$ and $v \in k^{*}$ (recall $v$ is the vertex such that $\alpha_1 = \alpha_v$). 
	
	If $v \in \mathcal{C}^+(k^{*})$, then $k^* \in M$. However, as $\bar{k} \in \beta$ and $\bar{k}^* \notin \beta$, $\beta$ does not contain every halfspace in $\bar{M}$. This contradicts our choice of $\beta$. 
	
	On the other hand, suppose $v \notin \mathcal{C}^+(k^{*})$. Let $\gamma$ be a geodesic from $\mathcal{C}^+(k^*)$ to $v$. Let $\hat{z}$ be the hyperplane that intersects the last edge of $\gamma$. By Lemma \ref{lemma_convexity2}, $\hat{z}$ does not intersect $\hat{k}$. Let $z$ be the choice of halfspace for $\hat{z}$ such that $v \in z$. It follows that $z \in M$. However, by the consistency condition, $\bar{k}^* \in \beta$. Again, this is a contradicts the choice condition on ultrafilters as we also have that $\bar{k} \in \beta$. Thus, $\alpha$ cannot be the limit of a sequence of non-terminating ultrafilters.
\end{proof}

We consider next the Euclidean case.

\begin{lemma} \label{lemma_euclidean_case}
	Let $X$ be an essential, cocompact, locally finite, Euclidean CAT(0) cube complex of Euclidean dimension $n$, then:
	\begin{enumerate}
		\item If $n = 1$, $\partial X = B(X)$, $X$ has caged hyperplanes and $X$ does not contain a tight cage.
		\item If $n > 1$, then $X$ is the product of two infinite, essential cube complexes.
	\end{enumerate}
\end{lemma}
\begin{proof}
	Let $F$ be an $\text{Aut}(X)$ invariant $n$-dimensional flat. As $X$ is cocompact and $F$ is $\text{Aut}(X)$ invariant, there is an $R>0$ such that $X$ is contained in the $R$ neighborhood of $F$. It also follows that every hyperplane intersects $F$. For if some hyperplane $\hat{h}$ did not intersect $F$, as $X$ is essential and $\hat{h}$ separates $X$, this would mean there are points in $X$ arbitrarily far from $F$. 
	
	We first suppose that $F$ is of dimension $n \ge 2$. As $X$ is contained in a bounded neighborhood of a flat of dimension $n$, it follows that $X$ does not contain any contracting isometries. By a version of the rank rigidity theorem for cube complexes \cite[Corollary 6.4]{CS}, $X$ is the product of two unbounded cube complexes. Moreover, as $X$ is essential, so are the factors in this product.
	
	Suppose now that $F$ is $1$-dimensional. In other words, $F$ is an $\text{Aut}(X)$ invariant CAT(0) geodesic line.

	We first claim that no hyperplane in $X$ has infinite diameter. For suppose some hyperplane $\hat{k}$ did have infinite diameter. 
	Let $m$ be an even integer larger than $4R$. As $X$ is essential and finite dimensional, it readily follows from Ramsey's Theorem (see \cite[Lemma 2.1]{CS}), that there exist halfspaces $k_1 \supset k_2\dots\supset k_{2m}$ of $X$ and corresponding hyperplanes $\hat{k}_1, \dots, \hat{k}_{2m}$ such that $\hat{k}= \hat{k}_m$.
 	Note that $\hat{k}$ does not intersect the $R$ neighborhood of $F \cap k_1^*$, as at least $R$ hyperplanes separate $\hat{k}_1$ from $\hat{k}$. Similarly, $\hat{k}$ does not intersect the $R$ neighborhood of $F \cap k_{2m}$. By the first paragraph, every hyperplane is in the $R$ neighborhood of $F$. Thus, $\hat{k}$ is contained in an $R$ neighborhood of the line segment of $F$ between $F \cap \hat{k}_1$ and  $F \cap \hat{k}_{2m}$ (recall that every hyperplane intersects $F$ by the first paragraph). However, this is a contradiction as $X$ is locally finite and $\hat{k}$ is infinite. Hence, every hyperplane in $X$ must have finite diameter. 
	
	Now consider an ultrafilter $\alpha \in \partial X$. We claim that $\alpha$ does not contain a minimal halfspace. For suppose otherwise that $h$ is a minimal halfspace in $\alpha$. As $\alpha \in \partial X$, there exists an infinite chain of halfspaces, $l_1 \supset l_2 \supset l_3 \supset \dots$ in $\alpha$. By Lemma \ref{lemma_intersects_inf_many} and since $h$ is minimal, $\hat{h}$  intersects $\hat{l}_i$ for infinitely many $i$. However, $X$ is locally finite and every hyperplane in $X$ has finite diameter, so this is not possible.
	
	It follows that every ultrafilter in $\partial X$ is non-terminating and that $\partial X = B(X)$. Furthermore, given any finite set of halfspaces $H$ in $X$, $Y = \bigcap_{h \in H} \mathcal{C}^+(h)$ is either empty or bounded (since hyperplanes are finite diameter). Thus, there cannot be a tight cage in $X$ as there cannot be a corresponding unbounded core $Y$.
	
	We next show $X$ has caged hyperplanes. For let $K$ be a cage in some ultrafilter $\alpha \in \partial X$. As $\alpha \in \partial X$, there exists an infinite chain of hyperplanes, $l_1 \supset l_2 \supset l_3 \dots$ in $\alpha$. By applying Lemma \ref{lemma_intersects_inf_many} and by possibly passing to a subsequence, we may assume for each $k \in K$, either $\hat{l}_i \cap \hat{k} \neq \emptyset$ or $l_i \subset k$. However, as $K$ is finite and hyperplanes in $X$ have finite diameter, for large enough $N$, $\hat{l}_N \subset k$ for every $k \in K$. Thus, as $K$ was an arbitrary cage, $X$ has caged hyperplanes.
\end{proof}

The proposition below summarizes the results in this section and immediately follows from Lemma \ref{lemma_reducible_case} and Lemma \ref{lemma_euclidean_case}.

\begin{proposition} \label{prop_red_euc_summary}
	Suppose $X$ is an essential, locally finite, cocompact CAT(0) cube complex. Furthermore, suppose either $X$ is reducible or $X$ is Euclidean. Then the following are equivalent:
	\begin{enumerate}
		\item $\partial X = B(X)$
		\item $X$ has caged hyperplanes
		\item $X$ is Euclidean with Euclidean dimension $1$
		\item $X$ does not contain a tight cage
	\end{enumerate}
\end{proposition}

\section{Characterizing when $\partial X = B(X)$} \label{sec_characterization}
Most of the work for the proof of Theorem \ref{thm_char} has been done in the previous sections. In this section, we piece together those results. 

The following lemma follows from results in \cite{NS} and gives one of the implications in Theorem \ref{thm_char} for the non-Euclidean, irreducible case.

\begin{lemma} \label{lemma_implies_caged}
	Let $X$ be a cocompact, essential, irreducible, locally finite, non-Euclidean CAT(0) cube complex. If $\partial X = B(X)$, then every open set $U \subset \mathcal{U}$ containing an element $\alpha \in \partial X$, contains a hyperplane.
\end{lemma}
\begin{proof}
	As $\partial X = B(X)$, $\alpha$ is an ultrafilter in $B(X)$. By \cite[Corollary 5.7]{NS} and as $X$ is essential, $U$ contains a hyperplane.
\end{proof}

\begin{theorem} \label{thm_char}
	Let $X$ be an essential, locally finite, cocompact, CAT(0) cube complex. The following are equivalent:
	\begin{enumerate}
		\item \label{thm_char_bdry} $\partial X = B(X)$
		\item \label{thm_char_caged_hyps} $X$ has caged hyperplanes.
		\item \label{thm_char_open_sets} Every open set in $\mathcal{U}(X)$ which contains an ultrafilter in $\partial X$, contains a hyperplane. 
	\end{enumerate}
	Additionally, if $X$ has straight links then the above conditions are equivalent to:
	\begin{enumerate}
		 \setcounter{enumi}{3}
		\item \label{thm_char_no_tight_cage} $X$ does not contain a tight cage.
	\end{enumerate}
\end{theorem}
\begin{proof}
	The equivalence of conditions \ref{thm_char_open_sets} and \ref{thm_char_caged_hyps} follows from Proposition \ref{prop_caged_equiv}. Conditions \ref{thm_char_bdry}--\ref{thm_char_no_tight_cage} are all equivalent in the Euclidean and reducible case by Proposition \ref{prop_red_euc_summary}. We now assume that $X$ is non-Euclidean and irreducible. 
	
	By Proposition \ref{prop_caged_implies_bx}, condition \ref{thm_char_caged_hyps} implies \ref{thm_char_bdry}. By Lemma \ref{lemma_implies_caged}, condition \ref{thm_char_bdry} implies \ref{thm_char_open_sets}. Finally, if $X$ has straight links then conditions \ref{thm_char_caged_hyps} and \ref{thm_char_no_tight_cage} are equivalent by Proposition \ref{prop_tight cage}.
\end{proof}

The following corollary immediately follows by combining the above theorem and Lemma \ref{prop_core_links}. The corollary gives a condition that is often easy to check in practice. This condition, for instance, is utilized in the examples of Section \ref{sec_applications}.

\begin{corollary} \label{cor_convex_join_subset}
	Let $X$ be a cocompact, locally finite CAT(0) cube complex with straight links. If $\partial X \neq B(X)$, then $X$ contains an unbounded convex subset $Y$ such that the link, taken in $X$, of every vertex in $Y$ is a join.
\end{corollary}

\section{Applications} \label{sec_applications}

Let $X$ be a locally finite, cocompact CAT(0) cube complex. By the above work, it turns out a good strategy to show that $\partial X = B(X)$ is to show that $X$ has straight links and does not contain an unbounded convex subcomplex whose vertices each have joins as their link (where the link is taken in $X$). This strategy is often straightforward to implement, and it is readily seen that many well studied irreducible CAT(0) cube complexes have Roller boundary equal to $B(X)$. We illustrate this approach in the case $X$ is the usual CAT(0) cube complex associated to a right-angled Coxeter group or right-angled Artin group.

\subsection{Right-angled Coxeter groups}

Given a simplicial graph $\Gamma$ with vertex set $S = \{s_1, s_2, ..., s_n\}$ and edge set $E$, the corresponding right-angled Coxeter group is given by the presentation:
\[W_{\Gamma} = \langle S ~| ~ s_i^2 = 1 \text{ for } 1 \le i \le n, s_is_j = s_js_i \text{ for } (s_i, s_j) \in E  \rangle \]
Every right-angled Coxeter group, $W_{\Gamma}$, acts geometrically on a CAT(0) cube complex, $\Sigma_\Gamma$, whose edges are labeled by vertices of $\Gamma$. This complex is known as the Davis complex. The $1$--skeleton of the link of every vertex in $\Sigma_{\Gamma}$ is isomorphic to $\Gamma$. Moreover, the labels of edges dual to a given hyperplane of $\Sigma_\Gamma$ are all the same vertex of $\Gamma$, and this vertex is called \textit{type} of the given hyperplane. If two hyperplanes of $\Sigma_\Gamma$ intersect, then their types are adjacent vertices of $\Gamma$. We refer the reader to \cite{Dav} or \cite{Dani} for a background on right-angled Coxeter groups.

In order to apply Corollary \ref{cor_convex_join_subset}, we first characterize when the Davis complex, $\Sigma_{\Gamma}$, has straight links and is essential. We note that the equivalence of (\ref{prop_racg_straight_links1}) and (\ref{prop_racg_straight_links3}) below is also proved in the author's thesis.

\begin{proposition} \label{prop_racg_straight_links}
	Let $W_{\Gamma}$ be a right-angled Coxeter group and $X = \Sigma_{\Gamma}$ be its Davis complex. Then the following are equivalent:
	\begin{enumerate}
		\item \label{prop_racg_straight_links1} The graph $\Gamma^c$ (the complement graph) does not have an isolated vertex.
		\item \label{prop_racg_straight_links2} $X$ has straight links.
		\item \label{prop_racg_straight_links3} $X$ is essential.
	\end{enumerate}
\end{proposition}
\begin{proof}
	The implication (\ref{prop_racg_straight_links2}) $\to$ (\ref{prop_racg_straight_links3}) follows from Lemma \ref{lemma_straight_links_implies_essential}. 
	
	We now suppose that $\Gamma^c$ does not have an isolated vertex and show that $X$ has straight links.
	Let $v$ be a vertex and $e$ an edge in $X$. Let $s \in \Gamma$ be the label of $e$. As $\Gamma^c$ does not have isolated vertices, there exists a vertex $t \in \Gamma$ that is not adjacent to $s$. Furthermore, there exists an edge, $f$, adjacent to $v$ in $X$ labeled by $t$. As $s$ and $t$ are not adjacent in $\Gamma$, the hyperplane dual to $e$ does not intersect the hyperplane dual to $f$. Thus, $X$ has straight links.
	
	Finally, we suppose that $X$ is essential and show that $\Gamma^c$ does not have isolated vertices. Let $s$ be any vertex of $\Gamma$, and let $e$ be the edge of $X$ that is adjacent to the identity vertex and which is labeled by $s$. Let $\hat{h}$ be the hyperplane dual to $e$, and let $h$ be the corresponding halfspace which contains the identity vertex. As $X$ is essential, there exists a vertex $u \in h$ that is distance $1$ from $\mathcal{C}^+(h)$. In particular, the hyperplane dual to the edge $f$ between $u$ and $\mathcal{C}^+(h)$ does not intersect $\hat{h}$. Thus, the label of $f$ is not adjacent to $s$ in $\Gamma$ (or else we would have that $f \subset \mathcal{C}^+(h)$). Consequently, $v$ is not isolated in $\Gamma^c$.
\end{proof}

\begin{theorem} \label{thm_racg_app}
	Let $W_{\Gamma}$ be an infinite right-angled Coxeter group, and let $X = \Sigma_{\Gamma}$ be the corresponding Davis complex. Then $\partial X = B(X)$ if and only if $\Gamma$ is not a join.
\end{theorem}
\begin{proof}
	Suppose first that $\Gamma$ is a join, $\Gamma = \Gamma_1 \star \Gamma_2$. As $W_\Gamma$ is infinite, without loss of generality we may assume that $\Gamma_2$ contains two non-adjacent vertices. If some vertex $v$ of $\Gamma_2$ is adjacent to every other vertex of $\Gamma_2$, then $\Gamma$ has the join decomposition $\Gamma = (\Gamma_1 \cup v) \star (\Gamma_2 \setminus v)$ and $\Gamma_2 \setminus v$ still contains two non-adjacent vertices. Thus, we can also assume without loss of generality that $\Gamma_2^c$ does not contain any isolated vertices.  The complex $X$ is then the product of two CAT(0) cube complexes, $X = \Sigma_{\Gamma_1} \times \Sigma_{\Gamma_2}$. By Proposition \ref{prop_racg_straight_links}, $\Sigma_{\Gamma_2}$ is essential, and as $\Gamma_1$ is not empty, $\Sigma_{\Gamma_1}$ contains an edge. Thus, $\partial X \neq B(X)$ by Lemma \ref{lemma_reducible_case}.
	
	On the other hand, suppose $\Gamma$ is not a join. As $W_{\Gamma}$ is infinite, $\Gamma$ is not a single vertex. By Lemma \ref{prop_racg_straight_links} $X$ has straight links. The link of any vertex of $X$ is isomorphic to $\Gamma$, so in particular is not a join. Thus by Corollary \ref{cor_convex_join_subset}, $\partial X = B(X)$.
\end{proof}

We remark that if $W_{\Gamma}$ is finite, then it trivially follows that $B(X) = \partial X = \emptyset$.

\subsection{Right-angled Artin groups}
Given a simplicial graph $\Gamma$ with vertex set $S = \{s_1, s_2, ..., s_n\}$ and edge set $E$, the corresponding right-angled Artin group is given by the presentation:
\[W_{\Gamma} = \langle S ~| ~ s_is_j = s_js_i \text{ for } (s_i, s_j) \in E  \rangle \]
Every right-angled Artin group is the fundamental group of a natural cube complex known as its Salvetti complex. A right-angled Artin group acts geometrically on the universal cover of its Salvetti complex, which is a CAT(0) cube complex. We refer the reader to \cite{Wise} for a background. Like with the Davis complex, the links of vertices in the universal cover of a Salvetti complex are isomorphic. Types of hyperplanes are defined similarly to that of the Davis complex, and the types of intersecting hyperplanes are adjacent vertices.

\begin{lemma}\label{lemma_raag_straight_links}
	Let $A_{\Gamma}$ be a right-angled Artin group, and let $X$ be the universal cover of the corresponding Salvetti complex. Then $X$ has straight links and is essential.
\end{lemma}
\begin{proof}
	Let $v$ be a vertex and $e$ an edge in $X$. Let $s \in V(\Gamma)$ be the label of $e$. It follows there is an edge, $f \neq e$, adjacent to $v$ with label $s$. As the hyperplane dual to $e$ has the same label as the hyperplane dual to $f$, they do not intersect. Thus, $X$ has straight links and is essential by Lemma \ref{lemma_straight_links_implies_essential}.
\end{proof}

The proof of the following theorem is similar to that of Theorem \ref{thm_racg_app}.

\begin{theorem} \label{thm_raag_app}
	Let $A_{\Gamma}$ be a right-angled Artin group, and let $X$ be the universal cover of the corresponding Salvetti complex. Then $\partial X = B(X)$ if and only if $\Gamma$ is not a join. 
\end{theorem}
\begin{proof}
	If $\Gamma$ is a join, then $X$ splits as a product whose factors, by Lemma \ref{lemma_raag_straight_links}, are essential. Thus, by Proposition \ref{lemma_reducible_case}, $\partial X \neq B(X)$. On the other hand, if $\Gamma$ is not a join, then the link of every vertex of $X$ is not a join as well (the link of a vertex of $X$ is a join if and only if $\Gamma$ is a join). Thus, $\partial X = B(X)$ by Lemma \ref{lemma_raag_straight_links} and Corollary \ref{cor_convex_join_subset}.
\end{proof}

Theorem \ref{intro_thm_racg_raag} from the introduction now follows from Theorem \ref{thm_racg_app} and Theorem \ref{thm_raag_app}, as a right-angled Artin/Coxeter group splits as a direct product if and only if its defining graph is a join.

\section{$2$--dimensional CAT(0) Cube Complexes} \label{sec_2d_case}

In this section, we focus on $2$--dimensional CAT(0) cube complexes. The goal is to prove Theorem \ref{thm_2_dim_ccc}.

The strategy for proving Theorem \ref{thm_2_dim_ccc} is as follows. We first prove several nice properties regarding tight cages in the $2$--dimensional setting. For one, the existence of a tight cage implies the existence of a sectorless tight cage (a tight cage with empty sector). This is the content of Proposition \ref{prop_tight_cage_implies_sectorless} below. Additionally, a halfspace that is in a sectorless tight cage has its positive carrier equal to the core of the cage (Proposition \ref{prop_2d_tight cage}). 

We then describe the $(G, \hat{h})$--collapsing map, which produces a new CAT(0) cube complex by collapsing the carriers of hyperplanes in the same orbit class of a given hyperplane, $\hat{h}$. We show that when $\hat{h}$ is chosen appropriately this new complex inherits desirable properties of the original complex. To prove Theorem \ref{thm_2_dim_ccc}, we apply the $(G, \hat{h})$--collapsing map to produce a new CAT(0) cube complex with strictly fewer orbit classes of sectorless tight cages. By applying such maps finitely many times, we obtain the main theorem.

\subsection{Tight cages in dimension two}

\begin{definition}[Sectorless Tight Cage] \label{def_sectorless_tight cage}
	Let $X$ be a CAT(0) cube complex. A \textit{sectorless tight cage} in $X$ is a tight cage, $(\mathcal{S}, \mathcal{T})$ in $X$, satisfying $\mathcal{S} = \emptyset$.  We usually simply say $\mathcal{T}$ is a sectorless tight cage (leaving $\mathcal{S}$ out of the notation).
\end{definition}

\begin{remark} \label{rmk_sectorless_tight_cages_in_2d}
	Let $\mathcal{T}$ be a sectorless tight cage in a $2$--dimensional CAT(0) cube complex. It follows that distinct hyperplanes in $\hat{\mathcal{T}}$ do not intersect. For suppose $h, k \in \mathcal{T}$ were distinct halfspaces such that $\hat{h}$ and $\hat{k}$ intersect. Then $\mathcal{C}^+(h) \cap \mathcal{C}^+(k)$ consists of exactly one vertex (as $X$ is $2$-dimensional). However, this contradicts tight cage condition \ref{def_tight cage_unbounded_core}, as then the core of $\mathcal{T}$ cannot be unbounded. 
\end{remark}

\begin{proposition} \label{prop_tight_cage_implies_sectorless}
	Let $X$ be a $2$--dimensional, cocompact CAT(0) cube complex which contains a tight cage, $(\mathcal{S}, \mathcal{T})$. Then $X$ contains a sectorless tight cage $\mathcal{T}'$ such that $\mathcal{T} \subset \mathcal{T}'$.
\end{proposition}
\begin{proof}	
	We may assume $\mathcal{S}$ is nonempty. Recall that a $2$--dimensional cube complex has at most two pairwise intersecting hyperplanes. Thus, as hyperplanes in $\hat{\mathcal{S}}$ each intersect every hyperplane in $\hat{\mathcal{T}}$ (tight cage condition \ref{def_tight_cage_intersections}), $\mathcal{S}$ contains exactly one halfspace, $\mathcal{S} = \{s\}$.
	
	Let $Z = \bigcap_{t \in \mathcal{T}} \mathcal{C}^+(t)$. Let $\hat{H}$ be the set of hyperplanes in $X$ that intersect $Z$. In particular, $\hat{s} \in \hat{H}$. As $X$ is $2$--dimensional, hyperplanes in $\hat{H}$ do not pairwise intersect. Furthermore, $|\hat{H}|$ is infinite as the core $Y = Z \cap s$ of the tight cage is unbounded (tight cage condition \ref{def_tight cage_unbounded_core}).
	Let $H$ be the choice of halfspaces for hyperplanes in $\hat{H}$ so that for each $h \in H$ either $h \subset s$ or $s \subset h$. 
	
	Let $\hat{K}$ be the set of hyperplanes that are dual to an edge adjacent to $Z$ and intersect every hyperplane in $\hat{H}$. Again by the dimension of $X$, a hyperplane in $\hat{K}$ cannot intersect a hyperplane in $\mathcal{T}$. Consequently, hyperplanes in $\hat{K}$ do not intersect $Z$. Let $K$ be the choice of halfspaces for hyperplanes in $\hat{K}$ such that $Z \subset k$ for each $k \in K$. 
	
	As $X$ is $2$--dimensional, for every $h \in H$ the set $\mathcal{C}^+(h) \cap Z$ consists of a single vertex, $v_h$. Furthermore, for every $k \in K$ and $h \in H$, $\mathcal{C}^+(k) \cap \mathcal{C}^+(h) \cap Z = v_h$ by Helly's property. In particular, as $X$ is locally finite, $|K|$ is finite. 
	
	Let $\mathcal{T}' = \mathcal{T} \cup K$, and set $Y' = \bigcap_{t \in \mathcal{T}'} \mathcal{C}^+(t)$.	We will show that $\mathcal{T}'$ is a sectorless tight cage with core $Y'$ by checking the tight cage conditions in Definition \ref{def_tight_cage}. Note that tight cage condition \ref{def_tight_cage_intersections} is always vacuously true for sectorless tight cages.
	
	As $|\mathcal{T}'| \ge |\mathcal{T}| \ge 2$, tight cage condition \ref{def_tight cage_large_T} follows. Furthermore, by Helly's property, $Y' \cap \hat{h} \neq \emptyset$ for every $\hat{h} \in \hat{H}$. Thus, $Y'$ is unbounded, and tight cage condition \ref{def_tight cage_unbounded_core} follows.
	
	Finally, we show that every hyperplane dual to an edge adjacent to $Y'$ is either in $\mathcal{T}'$ or intersects every hyperplane in $\hat{\mathcal{T}}'$, i.e. tight cage condition \ref{def_tight cage_core_prop}. Let $\hat{m}$ be a hyperplane, that is not in $\mathcal{T}'$ and is dual to an edge adjacent to $Y'$. 
	
	By Lemma \ref{lemma_cocompact_action_on_cores} there exists an isometry, $g$ of $X$, which stabilizes $Y'$, stabilizes the halfspaces in $\mathcal{T}'$ and such that $g\hat{m}$ is dual to an edge adjacent to $Y' \cap s$ (this last fact follows since $Y' \cap s$ is unbounded). As $g$ stabilizes hyperplanes in $\mathcal{T}'$, $g \hat{m} \notin \mathcal{T}'$. 
	
	As $(\mathcal{S}, \mathcal{T})$ is a tight cage and $Y' \subset Y$, $g \hat{m}$ intersects every hyperplane in $\mathcal{T}$. Thus $g\hat{m} \in \hat{H}$. By definition, $g \hat{m}$ intersect every hyperplane in $\hat{K}$. Consequently $g \hat{m}$ intersects every hyperplane in $\hat{\mathcal{T}}' = \hat{\mathcal{T}} \cup \hat{K}$ as well. Hence, $\hat{m}$ intersects every hyperplane in $\hat{\mathcal{T}}'$. We have thus shown that $\mathcal{T}'$ is a sectorless tight cage.
\end{proof}

\begin{proposition} \label{prop_2d_tight cage}
	Let $X$ be a $2$--dimensional CAT(0) cube complex containing a sectorless tight cage $\mathcal{T}$, and let $Y = \bigcap_{t \in \mathcal{T} }\mathcal{C}^+(t)$ be the core of this tight cage. Then $Y = \mathcal{C}^+(t)$ for every $t \in \mathcal{T}$.
\end{proposition}
\begin{proof}
	Let $t \in \mathcal{T}$ and let $e$ be an edge in $\mathcal{C}^+(t)$ that is adjacent to $Y$. As $\mathcal{C}^+(t)$ is connected, to prove the claim it suffices to show that $e \subset Y$.
	
	Let $\hat{h}$ be the hyperplane dual to $e$. 
	The hyperplane $\hat{h}$ intersects $\hat{t}$, so is not in $\hat{\mathcal{T}}$ by Remark \ref{rmk_sectorless_tight_cages_in_2d}. Thus, $\hat{h}$ intersects $Y$ by tight cage condition \ref{def_tight cage_core_prop}. By Lemma \ref{lemma_convexity1}, $e$ is contained in $Y$.
\end{proof}

\begin{remark}
	The $2$--dimensional assumption in the above proposition is necessary. The sectorless tight cage shown on the left of Figure \ref{fig_tight_cage}, for instance, does not satisfy the conclusion of this proposition.
\end{remark}

We get the following two corollaries:

\begin{corollary} \label{cor_2d_tight cage}
	Let $\mathcal{T}$ be a sectorless tight cage in a $2$--dimensional CAT(0) cube complex, then all hyperplanes in $\mathcal{T}$ are isometric.
\end{corollary}

\begin{corollary} \label{cor_neighboring_tight cages}
	Let $\mathcal{T}$ and $\mathcal{T}'$ be sectorless tight cages in a $2$--dimensional CAT(0) cube complex, and let $Y$ and $Y'$ respectively be their cores. If $d(Y, Y') = 1$, then $\hat{\mathcal{T}} \cap \hat{\mathcal{T}}' \neq \emptyset$, and a hyperplane intersects $Y$ if and only if it intersects $Y'$.
\end{corollary}
\begin{proof}
	Let $e$ be an edge with endpoints on $Y$ and $Y'$, and let $\hat{h}$ be a hyperplane dual to $e$. By convexity of $Y$ and $Y'$, $\hat{h}$ does not intersect $Y$ or $Y'$. By tight cage condition \ref{def_tight cage_core_prop}, $\hat{h} \in \hat{\mathcal{T}} \cap \hat{\mathcal{T}'}$.	This proves the first claim. The second claim follows from Proposition \ref{prop_2d_tight cage}, as a hyperplane intersects $Y$ if and only if it intersects $\hat{h}$ if and only if it intersects $Y'$.
\end{proof}

The following is a converse to Proposition \ref{prop_2d_tight cage} which holds in all dimensions.
\begin{proposition} \label{prop_tight_cages_by_carriers}
	Let $X$ be an essential, locally finite CAT(0) cube complex. Let $h$ be a halfspace such that the corresponding hyperplane $\hat{h}$ is unbounded. Additionally, suppose for every halfspace $k$ such that $\mathcal{C}^+(k) \cap \mathcal{C}^+(h) \neq \emptyset$ and $\hat{k} \cap \hat{h} = \emptyset$, it follows that $\mathcal{C}^+(k) = \mathcal{C}^+(h)$. Then $h$ is contained in a sectorless tight cage.
\end{proposition}
\begin{proof}
	Let $K$ be the set of all halfspaces, $k$, for which $\mathcal{C}^+(k) \cap \mathcal{C}^+(h) \neq \emptyset$ and $\hat{k} \cap \hat{h} = \emptyset$. Let $\mathcal{T} = K \cup h$. Given a vertex $v \in \mathcal{C}^+(h)$, it follows that $v \in \mathcal{C}^+(t)$ for every $t \in \mathcal{T}$. As $X$ is locally finite, $\mathcal{T}$ is finite. We show $\mathcal{T}$ is a sectorless tight cage by checking the tight cage conditions. Tight cage condition \ref{def_tight_cage_intersections} is vacuously true.
	
	As $X$ is essential, there exists a vertex $v \in h$ distance $1$ from $\mathcal{C}^+(h)$. Let $e$ be the edge between $v$ and $\mathcal{C}^+(h)$, and let $\hat{k}$ be the hyperplane dual to $e$. By Lemma \ref{lemma_convexity2}, $\hat{k} \cap \hat{h} = \emptyset$. Let $k$ be the choice of halfspace for $\hat{k}$ so that $\mathcal{C}^+(h) \subset k$. It follows that $k \in K \subset \mathcal{T}$. Thus, $|\mathcal{T}| \ge 2$ and tight cage condition \ref{def_tight cage_large_T} holds.
	
	Set $Y = \bigcap_{t \in \mathcal{T}} \mathcal{C}^+(t) = \mathcal{C}^+(h)$. As $\hat{h}$ is unbounded, $Y$ is unbounded as well. This shows tight cage condition \ref{def_tight cage_unbounded_core} holds. 
	
	Finally, given any edge, $e$, adjacent to $Y$, by construction either the hyperplane dual to $e$ is in $\hat{\mathcal{T}}$ or $e \in \mathcal{C}^+(h) = Y$. If $e \in Y$, then the hyperplane dual to $e$ intersects every hyperplane in $\hat{\mathcal{T}}$. This shows tight cage condition \ref{def_tight cage_core_prop}.
\end{proof}

The following definition singles out hyperplanes that have exactly one of their corresponding halfspaces contained in a sectorless tight cage.

\begin{definition}[Loose Hyperplane]
	We say the hyperplane, $\hat{h}$, is a \textit{loose hyperplane}, if for some choice of halfspace $h$, $\mathcal{C}^+(h)$ is the core of a sectorless tight cage and $\mathcal{C}^+(h^*)$ is not the core of any sectorless tight cage.
\end{definition}

\begin{lemma} \label{lemma_contains_loose_hyp}
	Let $X$ be a $2$--dimensional CAT(0) cube complex that contains a sectorless tight cage. If $X$ is irreducible, then $X$ contains a loose hyperplane.
\end{lemma}
\begin{proof}
	To prove the claim we assume that $X$ does not contain a loose hyperplane and deduce that $X$ must be reducible. Let $h$ be a halfspace in $X$ that is contained in a sectorless tight cage. As there are no loose hyperplanes in $X$, by Proposition \ref{prop_2d_tight cage}, both $Y_1 = \mathcal{C}^+(h)$ and $Y_2 = \mathcal{C}^+(h^*)$ are cores of sectorless tight cages $\mathcal{T}_1$ and $\mathcal{T}_2$.
	
	Let $\hat{H}_{\perp}$ be the set of hyperplanes in $X$ that intersect $\hat{h}$, and let $\hat{H}_{\parallel}$ be the set of hyperplanes in $X$ that do not intersect $\hat{h}$. Given $\hat{k} \in \hat{H}_{\parallel}$, let $\gamma$ be a geodesic from $\mathcal{C}(\hat{h})$ to $\mathcal{C}(\hat{k})$. Let 
	\[\hat{k}_1, \hat{k}_2, \dots, \hat{k}_n \]
	be the sequence of hyperplanes that intersect $\gamma$ listed in order in which they intersect $\gamma$. By Lemma \ref{lemma_convexity2}, $\hat{k}_i$ does not intersect $\hat{h}$ for each $1 \le i \le n$. For each $i$, choose the halfspace $k_i$ corresponding to $\hat{k}_i$ so that:
	\[k_n \supset k_{n-1} \supset \dots \supset k_1 \supset \hat{h}\]
	
	As $\hat{k}_1$ does not intersect $\mathcal{C}(\hat{h})$, by tight cage condition \ref{def_tight cage_core_prop} either $k_1 \in \mathcal{T}_1$ or $k_1 \in \mathcal{T}_2$. By Proposition \ref{prop_2d_tight cage}, a hyperplane intersects $\hat{h}$ if and only it intersects $\hat{k}_1$. By applying this reasoning iteratively, we conclude that a hyperplane intersects $\hat{k}$ if and only if it intersects $\hat{h}$.
	
	As $\hat{k}$ is an arbitrary hyperplane in $\hat{H}_{\parallel}$, it follows that every hyperplane in $\hat{H}_{\parallel} \cup \hat{h}$ intersects every hyperplane in $\hat{H}_{\perp}$. Furthermore, every hyperplane in $X$ is contained in the set $(\hat{H}_{\parallel} \cup \hat{h}) \cup \hat{H}_{\perp}$. By \cite[Lemma 2.5]{CS}, $X$ is reducible.
\end{proof}

\subsection{The collapsing map and its properties}
Let $X$ be a CAT(0) cube complex, and let $\hat{h}$ be a hyperplane in $X$. We obtain a new cube complex, $X_{\hat{h}}$, by collapsing the carrier $\mathcal{C}(\hat{h})  \cong \hat{h} \times [0,1]$ to the positive carrier $\mathcal{C}^+(h)  \cong \hat{h} \times \{1\}$ by the usual projection map. We denote by $\rho_{\hat{h}}: X \to X_{\hat{h}}$ the natural projection map. We say $X_{\hat{h}}$ is the complex obtained by applying the \textit{$\hat{h}$--collapsing map}, $\rho_{\hat{h}}$, to $X$. Such a collapsing map is also described and used in \cite{NS}.

Suppose the group $G$ acts by isometries on the CAT(0) cube complex $X$, and let $\hat{h}$ be a hyperplane in $X$. We would like to define a new CAT(0) cube complex, $X_{G, \hat{h}}$, by  collapsing hyperplanes in the $G$--orbit of $\hat{h}$. Formally, we obtain $X_{G, \hat{h}}$ by first arbitrarily ordering all hyperplanes in the $G$--orbit of $\hat{h}$:
\[\hat{h}_1, \hat{h}_2, \hat{h}_3 \dots \]
We form a sequence of CAT(0) cube complexes, $X_0 = X, X_1, X_2, \dots$ where $X_i$, for $i \ge 1$, is obtained by applying the $\hat{h}_i$--collapsing map to $X_{i-1}$. Here by a slight abuse of notation we denote by $\hat{h}_i$ the hyperplane in $X_{i-1}$ that is the image of the hyperplane $\hat{h}_i$ in $X$. i.e., $\hat{h}_i$ in $X_{i-1}$ is equal to $\rho_{\hat{h}_{i-1}}(  \rho_{\hat{h}_{i-2}}(\dots \rho_{\hat{h}_1}(\hat{h}_i)\dots $). We obtain the following sequence:
\[X \to_{\rho_{\hat{h}_1}} X_1 \to_{\rho_{\hat{h}_2} } X_2 \to_{\rho_{\hat{h}_3}} \dots \]

Define $X_{G, \hat{h}}$ to be the direct limit of this sequence, and define $\rho_{G, \hat{h}}: X \to X_{G, \hat{h}}$ to be the natural projection map. We call $\rho_{G, \hat{h}}$ the \textit{$(G, \hat{h})$--collapsing map}. When $\hat{h}$ and $G$ are understood, we set $\bar{X} = X_{G, \hat{h}}$ and $\rho = \rho_{G,\hat{h}}$. 

The following lemma gives some basic facts regarding this construction. The facts presented in this lemma will be used throughout this section.
\begin{lemma} \label{lemma_collapsed_basics}
	Let $X$ be a CAT(0) cube complex and $\hat{h}$ a hyperplane in $X$. Let $\rho = \rho_{G, \hat{h}}$ be the $G$--equivariant collapsing map, and let $\bar{X} = X_{G, \hat{h}}$ (as above), then:
	\begin{enumerate}
		\item \label{lemma_collapsed_basics_cat0} $\bar{X}$ is a CAT(0) cube complex.
		\item \label{lemma_collapsed_basics_hyp1} Given a hyperplane, $\hat{k}$, in $X$ that is not in the $G$--orbit of $\hat{h}$, it follows that $\rho(\hat{k})$ is a hyperplane in $\bar{X}$.
		\item \label{lemma_collapsed_basics_hyp2} Given a hyperplane $\hat{z}$ in $\bar{X}$, there exists a unique hyperplane $\hat{k}$ in $X$ such that $\hat{z} = \rho(\hat{k})$. We say that $\hat{k}$ is the \textbf{lift} of $\hat{z}$.
		\item \label{lemma_collapsed_basics_hyp3} Given two hyperplanes $\hat{k}$ and $\hat{k}'$ in $X$, each not in the $G$--orbit of $\hat{h}$, then $\rho(\hat{k})$ intersects $\rho(\hat{k}')$ if and only if $\hat{k}$ intersects $\hat{k}'$.
		\item \label{lemma_collapsed_basics_action} $G$ acts by isometries on $\bar{X}$, and $\rho$ is $G$--equivariant under this action.
	\end{enumerate}
\end{lemma}
\begin{proof}
	As above, let $h_1, h_2, \dots$ be an ordering on hyperplanes in the $G$--orbit of $\hat{h}$, and let $X = X_0, X_1, X_2, \dots$ be such that $X_i$ is obtained by applying the $\hat{h}_i$--collapsing map to $X_{i-1}$.
	
	\textbf{Proof of \ref{lemma_collapsed_basics_cat0} }: As $X_i$ is homotopy equivalent to $X_{i-1}$ for each $i \ge 1$, it is clear that $X_i$, for each $i$, and  $\bar{X}$ are simply connected. 
	
	We now show by induction, that for all $i \ge 0$, the link of every vertex in $X_i$ is a flag simplicial complex CAT(0). This combined with the previous paragraph imply that $X_i$ is a CAT(0) cube complex.  The base case, $X_0 = X$ follows by assumption. Assume now that $X_n$ is CAT(0) for some $n \ge 1$.
	
	Let $v$ be a vertex in $X_{n+1}$. If the preimage of $v$ in $X_n$ consists of a single vertex, it follows that the link of $v$ is isomorphic to the link of its preimage. We are then done in this case by the induction hypothesis. 
	
	The other possibility is that the preimage of $v$ consists of two vertices, $v_1$ and $v_2$ that are endpoints of an edge $e$ dual to the hyperplane $\hat{h}_{n+1}$. Let $\Delta_i$, for $i = 1,2$, be the simplicial complex which is the link of $v_i$. Let $u_i$ be the vertex of $\Delta_i$ that is dual to the edge $e$. Let $U_i$ be the vertices in $\Delta_i$ adjacent to $u_i$. Let $S_i$ be the simplicial complex spanned by $U_i$. It follows that $S_1$ is isomorphic to $S_2$, as these are symmetrical images of each other which lie on the subsets $\hat{h}_{n+1} \times \{0\}$ and $\hat{h}_{n+1} \times \{1\}$ of the carrier $\mathcal{C}(\hat{h}_{n+1})$. Let $\Delta_i'$ consist of $\Delta_i$ with $u_i$ removed and every simplex that contains $u_i$ removed as well. Let $\Delta$ denote the link of $v$. It follows $\Delta$ is isomorphic to the union of $\Delta_1'$ and $\Delta_2'$ identified along $S_1$ and $S_2$. 
	
	By this description, it follows the $1$--skeleton of $\Delta$ does not contain a bigon or loop (an edge whose endpoints are the same vertex). Suppose $u_1, \dots, u_m$ are vertices in $\Delta$ that form a clique. It follows that there are vertices $u_1', \dots, u_m'$ that are preimages of  $u_1, \dots, u_m$ and are all contained in either $\Delta_1'$ or $\Delta_2'$. Hence, $u_1', \dots ,u_m'$ spans a simplex and the image of this simplex is a simplex in $\Delta$. Thus, $X_{n+1}$ is CAT(0). We have thus shown that $X_i$ is a CAT(0) cube complex for all $i \ge 0$.
	
	Let $v$ now be a vertex in $\bar{X}$, and let $\Delta$ be its link. The $1$--skeleton of $\Delta$ cannot contain a bigon or loop. For then for some $i$, $X_i$  would contain a preimage of $v$ whose link contains a bigon or loop. Suppose now the vertices $u_1, \dots, u_m \in \Delta$ form a clique. Then there exist preimages $u_1', \dots, u_m'$ in $X_i$, for some $i$, which form a clique in $X_i$. As $X_i$ is CAT(0), $u_1', \dots, u_m'$ spans a simplex. The image under $\rho$ of this simplex is a simplex, as images of edges which contain the $u_i'$ are never collapsed. Thus, $\bar{X}$ is CAT(0).
	
	The proofs of \textbf{\ref{lemma_collapsed_basics_hyp1} - \ref{lemma_collapsed_basics_hyp3}} are straightforward checks.
	
	\textbf{Proof of \ref{lemma_collapsed_basics_action}:} Given $g \in G$ and a vertex $\bar{v} \in \bar{X}$ we define the action $g \bar{v} = \rho(gv)$ where $v \in X$ is a vertex in the preimage (under $\rho$) of $\bar{v}$. This action is well defined as given vertices $u, v \in X$, it follows that $\rho(v) = \rho(u)$ if and only if for all $g \in G$, $gu$ and $gv$ are separated only by hyperplanes in the $G$-orbit of $X$ if and only if $\rho(gu) = \rho(gv)$ for all $g \in G$. This action is defined on cells of $\bar{X}$ similarly. We now check that this gives an action by isometries.
	
	Let $u$ and $v$ be vertices in $\bar{X}$, and let $u'$ and $v'$ be vertices in $X$ in the preimage of respectively $u$ and $v$. Let $\gamma$ be a geodesic in $X$ from $u$ to $v$. Given a hyperplane $\hat{h}$ that intersects $\rho(\gamma)$, the lift (as in \ref{lemma_collapsed_basics_hyp2}) $\hat{h}'$ in $X$ of $\hat{h}$ intersects $\gamma$. Thus, no hyperplane intersects $\rho(\gamma)$ twice as no hyperplane intersects $\gamma$ twice. Hence, $\rho(\gamma)$ is a geodesic. 
	
	The geodesics $\gamma$ and $g \gamma$ contain the same number of hyperplanes that are not in the $G$--orbit of $h$. Thus, $\rho(\gamma)$ and $\rho(g \gamma)$ have the same length.
\end{proof}
We remark that $\bar{X}$ and $\rho$ also do not dependent on the ordering chosen for hyperplanes in the $G$--orbit of $\hat{h}$. However, this fact is not needed in our arguments. 

We now introduce the notion of carrier reflections, a hypothesis on a group action which will allow us to prove the main theorem of this section.

\begin{definition}[Carrier reflection] \label{def_carrier_reflection}
	Suppose $G$ acts on the $2$--dimensional CAT(0) cube complex $X$. Let $h$ be a halfspace in $X$ and $g \in G$. We say that $g$ is a \textit{carrier reflection of $h$}, if $g$ stabilizes $\mathcal{C}^+(h)$ and $g$ does not stabilize $\hat{h}$. We say \textit{$G$ acts without carrier reflections} if given any $g \in G$ and any halfspace $h$ in $X$, it follows that $g$ is not a carrier reflection of $h$. We say \textit{$G$ acts without core carrier reflections}, if given any $g \in G$ and any halfspace $h$ in $X$ which is contained in a sectorless tight cage, then it follows that $g$ is not a carrier reflection of $h$.
\end{definition}

We now show certain properties are preserved by the $(G, \hat{h})$--collapsing map.

\begin{lemma} \label{lemma_collapse_sequential_hyps}
	Let $X$ be a $2$--dimensional, locally finite CAT(0) cube complex. Suppose $G$ acts on $X$ without core carrier reflections. Let $\hat{h}$ be a loose hyperplane in $X$. Suppose $h_1, h_2$ and $h_3$ are halfspaces in $X$ such that $h_1 \supset h_2 \supset h_3$, $\mathcal{C}^+(h_1) \cap \mathcal{C}^+(h_2^*) \neq \emptyset$ and $\mathcal{C}^+(h_2) \cap \mathcal{C}^+(h_3^*) \neq \emptyset$. Then $\hat{h}_1, \hat{h}_2$ and $\hat{h}_3$ cannot all be in the $G$--orbit of $\hat{h}$.
\end{lemma}
\begin{proof}
	As $\hat{h}$ is loose, we can set $h$ to be the choice of halfspace for $\hat{h}$ that is contained in a sectorless tight cage and $h^*$ the choice of halfspace that is not contained in any sectorless tight cage.
	
	Suppose, for a contradiction, that there exist $g_1, g_2, g_3 \in G$ such that $\hat{h}_1 = g_1\hat{h}$, $\hat{h}_2 = g_2 \hat{h}$ and $\hat{h}_3 = g_3 \hat{h}$. Then for either $i=1$ or $i =3$, $g_2^{-1}\hat{h}_i = g_2^{-1}g_i\hat{h}$ is dual to an edge adjacent to $\mathcal{C}^+(h)$. Without loss of generality, suppose this is true for $i = 1$ and set $\hat{k} = g_2^{-1}g_1\hat{h}$. Let $k$ be the halfspace corresponding to $\hat{k}$ so that $\mathcal{C}^+(h) \subset k$.
	
	As $Y = \mathcal{C}^+(h)$ is the core of a sectorless tight cage and since $\hat{k}$ does not intersect $\hat{h}$, by tight cage condition \ref{def_tight cage_core_prop} and Proposition \ref{prop_2d_tight cage}, it follows that, $\mathcal{C}^+(k) = \mathcal{C}^+(h)$. 
	
	As $\mathcal{C}^+(h^*)$ is not the core of a sectorless tight cage and isometries permute cores of tight cages, it follows that $g_2^{-1}g_1$ stabilizes $\mathcal{C}^+(h)$ and does not stabilize $\hat{h}$. However, this contradicts our hypothesis that $G$ acts without core carrier reflections.
\end{proof}

\begin{lemma} \label{lemma_not_cores_preserved}
	Let $X$ be a $2$--dimensional, essential, locally finite CAT(0) cube complex. Suppose $G$ acts on $X$ without core carrier reflections. Let $\hat{h}$ be a loose hyperplane in $X$. Let $\rho = \rho_{G, \hat{h}}$ be the $(G, \hat{h})$--collapsing map and $\bar{X} = X_{G, \hat{h}}$. 
	
	Let $k$ be a halfspace in $X$ which is not contained in a sectorless tight cage and whose corresponding hyperplane, $\hat{k}$, is unbounded. Then there exists a vertex $v \in \mathcal{C}^+(k)$ such that $\rho(v)$ is not contained in the core of any sectorless tight cage in $\bar{X}$.
\end{lemma}
\begin{proof}
	Set $x_1 = k$. As $x_1$ is not in a sectorless tight cage, by Proposition \ref{prop_tight_cages_by_carriers} there exists a halfspace $x_2$ such that $\mathcal{C}^+(x_1) \cap \mathcal{C}^+(x_2) \neq \emptyset$, $\hat{x}_1 \cap \hat{x}_2 = \emptyset$ and $\mathcal{C}^+(x_1) \neq \mathcal{C}^+(x_2)$. 
	
	In particular, there exists an edge either in $\mathcal{C}^+(x_1) \setminus \mathcal{C}^+(x_2)$ or in $\mathcal{C}^+(x_2) \setminus \mathcal{C}^+(x_1)$. Without loss of generality, suppose $e$ is an edge in $\mathcal{C}^+(x_2) \setminus \mathcal{C}^+(x_1)$ (the proof in the other case is the same). We may further suppose that $e$ is adjacent to a vertex, $v$ in $\mathcal{C}^+(x_1)$. Let $\hat{x}_3$ be the hyperplane dual to $e$. By Lemma \ref{lemma_convexity1}, $\hat{x}_3$ does not intersect $\hat{x}_1$. Let $x_3$ be any choice of halfspace for $\hat{x}_3$.
	
	For $1 \le i \le 3$, we define the halfspace $x_i'$. If $\hat{x}_i$ is not in the $G$--orbit of $\hat{h}$, set $x_i' = x_i$. On the other hand, suppose that $\hat{x}_i$ is in the $G$--orbit of $\hat{h}$. As $h$ is loose, there exists a halfspace, $y_i$, such that $y_i$ is in the same sectorless tight cage as either $x_i$ or $x_i^*$. In this case, we set $x_i' = y_i$. Note that $y_i$ cannot be in the $G$--orbit of $\hat{h}$ since $G$ acts without core carrier reflections.
	
	By Proposition \ref{prop_2d_tight cage}, a hyperplane intersects $\hat{x}_i$ if and only if it intersects $\hat{x}_i'$. Set $\bar{x}_i = \rho(\hat{x}_i')$. By construction $\hat{x}_i'$ is never in the $G$--orbit of $\hat{h}$ and $\bar{x}_i$ is a hyperplane in $\bar{X}$. Additionally, by construction $\rho(v) \in \bigcap_{i = 1}^{i = 3} \mathcal{C}^+(\bar{x}_i)$. Finally, we have that $\bar{x}_1$ does not intersect $\bar{x}_2$, $\bar{x}_3$ intersects $\bar{x}_2$ and $\bar{x}_3$ does not intersect $\bar{x}_1$.
	
	Suppose, for a contradiction, that $\rho(v)$ is contained in the core of a sectorless tight cage $\mathcal{T}$. By tight cage condition \ref{def_tight cage_core_prop}, every hyperplane dual to an edge adjacent to $\rho(v)$ is either in $\hat{\mathcal{T}}$ or intersects every hyperplane in $\hat{\mathcal{T}}$. 
	
	Suppose first that $\bar{x}_1 \in \hat{\mathcal{T}}$. Then $\bar{x}_2 \in \hat{\mathcal{T}}$ since $\bar{x}_1$ and $\bar{x}_2$ do not intersect. However, $\bar{x}_3 \notin \hat{\mathcal{T}}$ since $\bar{x}_3$ intersects $\bar{x}_2$ and no pair of distinct hyperplanes in $\hat{\mathcal{T}}$ intersect (Remark \ref{rmk_sectorless_tight_cages_in_2d}). On the other hand, $\bar{x}_3$ does not intersect every hyperplane in $\hat{\mathcal{T}}$ since $\bar{x}_3$ does not intersect $\bar{x}_1$. This is a contradiction.
	
	On the other hand, suppose that $\bar{x}_1 \notin \hat{\mathcal{T}}$. Then $\bar{x}_2, \bar{x}_3 \notin \hat{\mathcal{T}}$ since $\bar{x}_1$ does not intersect $\bar{x}_2$ or $\bar{x}_3$. It follows that both $\bar{x}_2$ and $\bar{x}_3$ intersect some hyperplane $\hat{t} \in \hat{\mathcal{T}}$. However, this is a contradiction, as there cannot be three pairwise intersecting hyperplanes in a $2$--dimensional cube complex. Thus, the claim follows.
\end{proof}

\begin{lemma} \label{lemma_core_lifts}
	Let $X$ be a $2$--dimensional, essential, locally finite CAT(0) cube complex. Suppose $G$ acts on $X$ without core carrier reflections. Let $\hat{h}$ be a loose hyperplane in $X$. Let $\rho = \rho_{G, \hat{h}}$ be the $(G, \hat{h})$--collapsing map and $\bar{X} = X_{G, \hat{h}}$.
	
	Suppose $k_1$ and $k_2$ are distinct halfspaces in $\bar{X}$ contained in a common sectorless tight cage. Let $l_1$ and $l_2$ be their lifts in $X$. Then $l_1$ and $l_2$ are contained in a common sectorless tight cage of $X$.
\end{lemma}
\begin{proof}
	Note that $\hat{l}_1$ and $\hat{l}_2$ are unbounded since $\hat{k}_1$ and $\hat{k}_2$ are unbounded (tight cage condition \ref{def_tight cage_unbounded_core}).
		
	We first claim that $l_1$ is contained in some sectorless tight cage, $\mathcal{T}$, in $X$. For if it were not, by Lemma \ref{lemma_not_cores_preserved} there exists a vertex $v \in \mathcal{C}^+(l_1)$ such that $\rho(v)$ is not contained in the core of a sectorless tight cage. However, this is not possible as $\rho(v)$ is contained the core, $\mathcal{C}^+(k_1)$.
		
	Suppose now that $\mathcal{C}^+(l_1) \cap \mathcal{C}^+(l_2) \neq \emptyset$. In this case we can apply tight cage condition \ref{def_tight cage_core_prop} to $\mathcal{T}$ to conclude that $l_2 \in \mathcal{T}$ (as $\hat{l}_1 \cap \hat{l}_2 = \emptyset$). The lemma then follows in this case.
		
	On the other hand, suppose that $d(\mathcal{C}^+(l_1), \mathcal{C}^+(l_2)) > 0$. By Lemma \ref{lemma_convexity2}, there exists a hyperplane, $\hat{l}$, such that $\hat{l}$ separates $\hat{l}_1$ from $\hat{l}_2$ and $\mathcal{C}^+(l) \cap \mathcal{C}^+(l_1) \neq \emptyset$. Since $\hat{l}$ separates $\hat{l}_1$ from $\hat{l}_2$, $\hat{l}$ must be in the $G$--orbit of $\hat{h}$. Let $l$ be the choice of halfspace for $\hat{l}$ so that $\hat{l}_1 \subset l$. By again applying tight cage condition \ref{def_tight cage_core_prop}, we conclude that $l \in \mathcal{T}$. 
	
	As $\hat{h}$ is loose, $l^*$ is not in a sectorless tight cage. By Lemma \ref{lemma_not_cores_preserved}, there exists a vertex $v \in \mathcal{C}^+(l^*)$ such that $\rho(v)$ is not in the core of any sectorless tight cage. However,this gives a contradiction since $\rho(v) \in \mathcal{C}^+(k_1)$. Thus, this second case is not possible. The lemma thus follows.
\end{proof}

The next proposition guarantees that $X_{G, \hat{h}}$ inherits many nice properties of $X$ when $\hat{h}$ is a loose hyperplane.

\begin{proposition} \label{prop_remains_nice}
	Let $X$ be a $2$--dimensional, locally finite CAT(0) cube complex with straight links. Suppose $G$ acts properly, cocompactly and without core carrier reflections on $X$. Let $\hat{h}$ be loose hyperplane in $X$, $\rho = \rho_{G, \hat{h}}$ be the $(G, \hat{h})$--collapsing map and $\bar{X} = X_{G, \hat{h}}$.  Then 
	\begin{enumerate}
		\item \label{lemma_remains_nice_dimension} $\bar{X}$ is $2$--dimensional
		\item \label{lemma_remains_nice_loc_fin} $\bar{X}$ is locally finite
		\item \label{lemma_remains_nice_qi} $\rho$ is a quasi-isometry
		\item \label{lemma_remains_nice_proper_cocompact} $G$ acts properly and cocompactly on $\bar{X}$
		\item \label{lemma_remains_nice_straight_links} $\bar{X}$ has straight links
		\item \label{lemma_remains_nice_no_carrier_ref} $G$ acts on $\bar{X}$ without core carrier reflections
	\end{enumerate}
\end{proposition}
\begin{proof}
	
	\textbf{Proof of \ref{lemma_remains_nice_dimension}:}
	Given a set, $\hat{S}$, of pairwise intersecting hyperplanes in $\bar{X}$, there exist a set of pairwise intersecting hyperplanes in $X$ consisting of the lifts of hyperplanes in $\hat{S}$. As $X$ has dimension two, the dimension of $\bar{X}$ is at most two.
	
	\textbf{Proof of \ref{lemma_remains_nice_loc_fin}:}
	Let $\bar{v}$ be a vertex in $\bar{X}$, and let $V$ be the set of vertices in $X$ which lie in the preimage $\rho^{-1}(\bar{v})$. Let $Z$ be the set $V$ along with the edges connecting pairs of vertices in $V$. Note that each of these edges is dual to a hyperplane in the $G$--orbit of $\hat{h}$. Furthermore, the set $Z$ is connected.
	
	By an application of Ramsey's Theorem (see \cite[Lemma 2.1]{CS}), there exists a constant $M$ such that any geodesic of length $M$ crosses $3$ pairwise non-intersecting hyperplanes. We claim that $d(v, v') < M$ for every $v, v' \in V(\bar{v})$ (where distance is taken in $X$). 
	
	Suppose otherwise, that for some $v, v' \in V$, $d(v, v') \ge M$. Let $\gamma$ be a geodesic in $X$ from $v$ to $v'$. Every hyperplane dual to $\gamma$ is in the $G$--orbit of $\hat{h}$. As, $|\gamma| \ge M$, there exist three non-intersecting hyperplanes, $\hat{h}_1, \hat{h}_2$ and $\hat{h}_3$ that intersect $\gamma$. Without loss of generality, we can assume that no hyperplane separates any two of these hyperplanes. However, by Lemma \ref{lemma_collapse_sequential_hyps} this gives a contradiction.
	
	As $X$ is locally finite and $Z$ is finite diameter, $|V|$ is finite. Furthermore, the number of edges adjacent to a vertex in $V$ is finite as well. Thus, $\bar{X}$ is locally finite.
	
	\textbf{Proof of \ref{lemma_remains_nice_qi} }:
	Let $M$ be the same constant as above. Let $v, v'$ be vertices in $X$. Let $\gamma$ be a geodesic from $v$ to $v'$. By the above proof, at least $\lfloor |\gamma|/M \rfloor$ hyperplanes which intersect $\gamma$ are not in the $G$--orbit of $\hat{h}$. It follows that $d_X(v, v')/M - M \le d_{\bar{X}}(\rho(v), \rho(v'))$. The other inequality, $d_{\bar{X}}(\rho(v), \rho(v')) \le d_X(v, v')$, is obvious.
	
	\textbf{Proof of \ref{lemma_remains_nice_proper_cocompact}:}
	As $\bar{X}$ is locally finite and each vertex has finitely many preimages (by argument above), it follows that $G$ acts properly and cocompactly on $\bar{X}$ since $G$ acts properly and cocompactly on $X$.
	
	\textbf{Proof of \ref{lemma_remains_nice_straight_links}:}
	Let $\bar{v}$ be a vertex in $\bar{X}$, and let $\bar{e}$ be an edge adjacent to $\bar{v}$. Let $v \in X$ and $e$ be an edge adjacent to $v$ such that $\rho(v) = \bar{v}$ and $\rho(e) = \bar{e}$. The hyperplane, $\hat{k}_e$, dual to $e$ is not in the same orbit class as $\hat{h}$ (since $\rho(e)$ is an edge). 
	
	As $X$ has straight links, there exist edges $e_1, e_2$ and $e_3$ in $X$ such that the concatenation $e \cup e_1 \cup e_2 \cup e_3$ is a geodesic of length $4$, the hyperplanes $\hat{h}_e$, $\hat{h}_{e_1}$, $\hat{h}_{e_2}$ and $\hat{h}_{e_3}$ do not pairwise intersect (where $\hat{h}_{e_i}$ the hyperplane dual to $e_i$), and $e \cap e_1 = v$. 
	
	By Lemma \ref{lemma_collapse_sequential_hyps}, there exists $r$, $1 \le r \le 3$, such that $\hat{h}_{e_r}$ is not in the $G$--orbit of $\hat{h}$. Suppose further $r$ is minimal out of possible such choices. It follows that the edges $\bar{e}$ and $\rho(e_r)$ are adjacent to $\bar{v}$ and the hyperplanes dual to these edges do not intersect. Thus, the straight links condition is satisfied. 
	
	\textbf{Proof of \ref{lemma_remains_nice_no_carrier_ref}}:
	Suppose, for a contradiction, some $g \in G$ acts as a core carrier reflection on $\bar{X}$. In particular, there are halfspaces $k_1$ and $k_2$ in $\bar{X}$ such that $Y = \mathcal{C}^+(k_1) = \mathcal{C}^+(k_2)$ is the core of a sectorless tight cage, $g$ stabilizes $Y$ and $g \hat{k}_1 = \hat{k}_2$. 
	
	Let $l_1$ and $l_2$ be halfspaces in $X$ that are the lifts of $k_1$ and $k_2$. By Lemma \ref{lemma_core_lifts}, $l_1$ and $l_2$ are in a sectorless tight cage of $X$. As $gl_1 = l_2$, this implies $g$ acts as a core carrier reflection on $l_1$. However, this contradicts our hypothesis that $G$ acts without core carrier reflections on $X$.
\end{proof}

\begin{theorem} \label{thm_2_dim_ccc}
	Let $X$ be a $2$--dimensional, irreducible, locally finite CAT(0) cube complex with straight links. Suppose $G$ acts properly, cocompactly and without core carrier reflections on $X$. Then $G$ acts properly and cocompactly on a $2$--dimensional, irreducible, locally finite CAT(0) cube complex with straight links, $\bar{X}$, which satisfies $\partial \bar{X} = B(\bar{X})$. Furthermore, $X$ is $G$--equivariantly quasi-isometric to $\bar{X}$.
\end{theorem}
\begin{proof}
	If $X$ does not contain a tight cage, then by Theorem \ref{thm_char} we have that $\partial X = B(X)$. The claim then follows by setting $\bar{X} = X$. 
	
	On the other hand, suppose that that $X$ contains a tight cage. By Proposition \ref{prop_tight_cage_implies_sectorless}, $X$ contains a sectorless tight cage $\mathcal{T}$. Furthermore, by Lemma \ref{lemma_contains_loose_hyp}, there exists a loose hyperplane $\hat{h}$ in $X$. Set $X' = X_{G, \hat{h}} = \rho(X)$, where $\rho = \rho_{G, \hat{h}}$ is the $(G, \hat{h})$--collapsing map. By Proposition \ref{prop_remains_nice}, $X'$ is a $2$--dimensional, locally finite CAT(0) cube complex with straight links on which $G$ acts properly and cocompactly, and the map $\rho$ is a $G$--equivariant quasi-isometry. 
	
	As $\hat{h}$ is loose, we can orient the halfspace $h$ associated to $\hat{h}$ so that $C^+(h)$ is in a sectorless tight cage and $C^+(h^*)$ is not in any sectorless tight cage. As $C^+(h)$ is in a sectorless tight cage and the core of tight cages are unbounded, it follows that $\hat{h}$ is unbounded. By Lemma \ref{lemma_not_cores_preserved}, there exists a vertex $v \in C^+(h^*)$ such that $\rho(v)$ is not contained in any sectorless tight cage. Thus, $\rho(k)$, for any halfspace in $k \in \mathcal{T}$, is not contained in a sectorless tight cage. Consequently, by Lemma \ref{lemma_core_lifts} $X'$ has strictly less orbit classes of sectorless tight cages than $X$. 
	
	As $X$ has only finitely many orbit classes of sectorless tight cages, we can repeat this process finitely many times to obtain the cube complex $\bar{X}$ which does not contain any sectorless tight cage. The composition of the corresponding $(G, \hat{h})$--collapsing maps gives the desired $G$--equivariant quasi-isometry. 
	
	By Proposition \ref{prop_tight_cage_implies_sectorless}, $\bar{X}$ does not contain a tight cage. By Proposition \ref{prop_red_euc_summary}, $\bar{X}$ must be irreducible. Finally, by Theorem \ref{thm_char}, $\partial \bar{X} = B(\bar{X})$.
\end{proof}

\bibliographystyle{amsalpha}
\bibliography{mybibliography}
\end{document}